\documentclass{article}
\usepackage{arxiv}
\usepackage{wrapfig}
\usepackage{tikz}
\usepackage{graphicx}
\let\theoremstyle\relax
\usepackage{amsmath,amssymb,amsfonts,mathrsfs,amsthm}
\usepackage{stmaryrd}     %
\usepackage{adjustbox}
\usepackage{multirow}
\usepackage{algorithm}  %
\usepackage{algpseudocode}%

\usepackage{cite}       %
\usepackage{hyperref}

\date{}
\title{Projection-free computation of robust controllable sets with constrained zonotopes}

\author{Abraham P. Vinod\thanks{Corresponding author}, Avishai Weiss, and Stefano Di Cairano \\
	Mitsubishi Electric Research Laboratories, Cambridge, MA 02139 \\
	\texttt{abraham.p.vinod@ieee.org}, \texttt{weiss@merl.com}, \texttt{dicairano@ieee.org} \\
}

\begin{document}
\newcommand{\bbb}{\mathbb{B}}
\newcommand{\bbc}{\mathbb{C}}
\newcommand{\bbe}{\mathbb{E}}
\newcommand{\bbn}{\mathbb{N}}
\newcommand{\bbp}{\mathbb{P}}
\newcommand{\bbr}{\mathbb{R}}
\newcommand{\bbs}{\mathbb{S}}
\newcommand{\bbz}{\mathbb{Z}}
\newcommand{\bfc}{\mathbf{C}}
\newcommand{\bfx}{\mathbf{X}}
\newcommand{\cala}{\mathcal{A}}
\newcommand{\calb}{\mathcal{B}}
\newcommand{\calc}{\mathcal{C}}
\newcommand{\cald}{\mathcal{D}}
\newcommand{\cale}{\mathcal{E}}
\newcommand{\calf}{\mathcal{F}}
\newcommand{\calg}{\mathcal{G}}
\newcommand{\calh}{\mathcal{H}}
\newcommand{\cali}{\mathcal{I}}
\newcommand{\calk}{\mathcal{K}}
\newcommand{\call}{\mathcal{L}}
\newcommand{\calm}{\mathcal{M}}
\newcommand{\caln}{\mathcal{N}}
\newcommand{\calo}{\mathcal{O}}
\newcommand{\calp}{\mathcal{P}}
\newcommand{\calq}{\mathcal{Q}}
\newcommand{\calr}{\mathcal{R}}
\newcommand{\cals}{\mathcal{S}}
\newcommand{\calt}{\mathcal{T}}
\newcommand{\calu}{\mathcal{U}}
\newcommand{\calv}{\mathcal{V}}
\newcommand{\calw}{\mathcal{W}}
\newcommand{\calx}{\mathcal{X}}
\newcommand{\caly}{\mathcal{Y}}
\newcommand{\calz}{\mathcal{Z}}
\newcommand{\scra}{\mathscr{A}}
\newcommand{\scrb}{\mathscr{B}}
\newcommand{\scrc}{\mathscr{C}}
\newcommand{\scre}{\mathscr{E}}
\newcommand{\scrg}{\mathscr{G}}
\newcommand{\scrh}{\mathscr{H}}
\newcommand{\scrl}{\mathscr{L}}
\newcommand{\scro}{\mathscr{O}}
\newcommand{\scrp}{\mathscr{P}}
\newcommand{\scrr}{\mathscr{R}}
\newcommand{\scru}{\mathscr{U}}
\newcommand{\scrv}{\mathscr{V}}
\newcommand{\scrw}{\mathscr{W}}
\newcommand{\scrx}{\mathscr{X}}

\newcommand{\Nint}[2]{\mathbb{N}_{[#1:#2]}}

\newtheorem{theorem}{Theorem}
\newtheorem{assumption}{Assumption}
\newtheorem{condition}{Condition}
\newtheorem{remark}{Remark}
\newtheorem{lemma}{Lemma}
\newtheorem{problem}{Problem}
\newtheorem{corollary}{Corollary}
\newtheorem{proposition}{Proposition}
\newtheorem{definition}{Definition}
\newtheorem{example}{Example}

\newcommand{\dofo}[1]{\scro_{{#1}}}
\newcommand{\dvec}{\xi}
\newcommand{\nineq}[1]{{L_{#1}}}
\newcommand{\ndvec}[1]{{N_{#1}}}
\newcommand{\nconst}[1]{{M_{#1}}}
\newcommand{\convHull}[1]{\operatorname{CH}\left({#1}\right)}
\newcommand{\affHull}[1]{\operatorname{AH}\left({#1}\right)}
\newcommand{\ndvecEllipsoid}{n}
\newcommand{\Tplan}{T_\text{plan}}
\newcommand{\Tsafe}{T_\text{safe}}
\newcommand{\Tsim}{T_\text{sim}}
\newcommand{\Uoffnom}{\calu_\text{off-nom}}
\newcommand{\Woffnom}{\calw_\text{off-nom}}
\newcommand{\Eoffnom}{\cale_\text{off-nom}}
\newcommand{\Nmass}{{N_M}}

\maketitle

\begin{abstract} 
We study the problem of computing robust controllable sets for discrete-time linear systems with additive uncertainty. We propose a tractable and scalable approach to inner- and outer-approximate robust controllable sets using constrained zonotopes, when the additive uncertainty set is a symmetric, convex, and compact set. Our least-squares-based approach uses novel closed-form approximations of the Pontryagin difference between a constrained zonotopic minuend and a symmetric, convex, and compact subtrahend. We obtain these approximations using two novel canonical representations for full-dimensional constrained zonotopes. Unlike existing approaches, our approach does not rely on convex optimization solvers, and is projection-free for ellipsoidal and zonotopic uncertainty sets. We also propose a least-squares-based approach to compute a convex, polyhedral outer-approximation to constrained zonotopes, and characterize sufficient conditions under which all these approximations are exact. We demonstrate the computational efficiency and scalability of our approach in several case studies, including the design of abort-safe rendezvous trajectories for a spacecraft in near-rectilinear halo orbit under uncertainty. Our approach can inner-approximate a 20-step robust controllable set for a 100-dimensional linear system in under 15 seconds on a standard computer.
\end{abstract}

\section{Introduction}
 
Robust controllable (RC) sets characterize a set of states from which a collection of possibly time-varying state constraints can be satisfied by a controlled state trajectory, under bounded control authority and uncertainty. RC sets are essential for robust model predictive control~\cite{langson2004robust,mayne2005robust,borrelli2017predictive,bertsekas_1971}, fault-tolerant
control~\cite{scott_constrained_2016,raghuraman2022set,xu2024observer,GruberRobustSafetyTAC}, and verification of dynamical systems~\cite{yang_efficient_2022,gleason2021lagrangian,vinod2021abort,vinod2019sreachtools}, and have been used in a broad range of applications in space~\cite{daniel,vinod2021abort,gleason2021lagrangian,vinod2019sreachtools}, transportation~\cite{ahn2020reachability}, and robotics~\cite{malone2017hybrid,Blanchini2015}. For discrete-time linear systems with additive uncertainty and bounded, polytopic state and input constraints, exact RC sets are known to be polytopes, and the RC sets may be computed using set computations on polytopes. However, polytope-based RC set computation involves projection that has a combinatorial computational complexity and causes numerical issues for high-dimensional systems and/or over long horizons~\cite{Blanchini2015,borrelli2017predictive,althoff2021set,kerrigan2000robust}. In this paper, we focus in addressing these shortcomings, and propose \emph{novel theory and algorithms for efficient and scalable computation of inner- and outer-approximations of RC sets, that admit closed-form description of the sets involved and can be computed without relying on convex optimization solvers in several cases.}

We are motivated by the problem of designing \emph{active, abort-safe, spacecraft rendezvous} trajectories with the Lunar gateway~\cite{gateway}. In abort-safe rendezvous, we seek rendezvous trajectories that nominally steer the control-constrained spacecraft towards the rendezvous target, while allowing for the possibility of diverting away from the rendezvous target in the event of failure. Often, failures in space applications are characterized by partial or complete loss of actuation, and increased actuation and navigational uncertainties that may appear as additive uncertainties in the dynamics. For designing such rendezvous trajectories, we use RC sets to characterize a set of state constraints that a nominal trajectory should satisfy, where the RC sets encode the desirable property of safe abort under limited available actuation and uncertainty, similarly to~\cite{daniel,vinod2021abort}. However, polytope-based computation of RC sets for high-dimensional systems suffers from numerical issues, primarily due to projection~\cite{Blanchini2015,borrelli2017predictive,kerrigan2000robust,althoff2021set}.
On the other hand, our approach can compute the required high-dimensional RC sets. %

We will use constrained zonotopes, a recently proposed alternative description to polytopes~\cite{scott_constrained_2016,raghuraman2022set,yang_efficient_2022,Althoff2015ARCH}, to achieve tractable approximation of RC sets.
In the disturbance-free setting, constrained zonotopes provide closed-form expressions for all set operations used in the exact computation of the \emph{controllable set}~\cite{scott_constrained_2016}. However, an exact computation of RC sets using constrained zonotopes is hindered by the fact that there are no tractable approaches to compute the exact Pontryagin difference with a constrained zonotopic minuend~\cite{raghuraman2022set,yang_efficient_2022}.

Recently,~\cite{yang_efficient_2022} described an optimization-based two-stage approach to inner-approximate the Pontryagin difference between a constrained zonotope and a zonotope, which allows for computing inner-approximations of RC sets when the additive disturbance set is a zonotope. However, it is unclear how such an approach extends to ellipsoidal uncertainty (a more common setting in control), and the reliance of optimization hinders fast computation of RC sets for high-dimensional systems. In this paper, we propose \emph{least-squares-based algorithms to generate constrained zonotopes that inner- and outer-approximate the Pontryagin difference between a constrained zonotopic minuend and a symmetric, convex, and compact subtrahend}. Our approaches also admit closed-form expressions for the approximations when using a full-dimensional minuend, and an ellipsoidal or zonotopic subtrahend. Then, we use these algorithms for fast and scalable computation of RC sets for an additive uncertainty set that is symmetric, convex, and compact.

The main contributions of this paper are as follows: 1) a tractable inner- and outer-approximation of the Pontryagin difference between a constrained zonotopic minuend and a symmetric, convex, and compact subtrahend, 2) an inner- and outer-approximation of the RC sets using the proposed approximations to the Pontryagin difference, 3) a closed-form description of a polyhedral outer-approximation of a constrained zonotope, and 4) sufficient conditions under which the approximations proposed above are exact.
We propose two canonical representations for full-dimensional constrained zonotopes, which facilitate the use of the least-squares method as the primary tool for all of our approaches. 
Efficient implementations for the least-squares method are well-known~\cite{boyd_introduction_2018}.
Unlike~\cite{yang_efficient_2022}, the proposed approximations to the Pontryagin difference admit closed-form expressions when the subtrahend is an ellipsoid, a zonotope, or a convex hull of a collection of symmetric intervals. 
These features together enable an inner-approximation of a 20-step RC set for a 100-dimensional linear system in less than 15 seconds on a standard computer.

The rest of this paper is organized as follows: Sec.~\ref{sec:prelim} provides the necessary mathematical background and states the problem statements of interest. Sec.~\ref{sec:inner_approx} and~\ref{sec:outer_approx} describe the proposed approaches to approximate the Pontryagin difference between a constrained zonotopic minuend and a symmetric, convex, and compact subtrahend. Sec.~\ref{sec:inner_approx_RC} applies these approaches to the computation of RC sets, and Sec.~\ref{sec:num} presents several case studies that demonstrate the utility and scalability of the proposed approach. Sec.~\ref{sec:num_app} applies the method to the (simplified) motivating problem of active abort-safe rendezvous with the Lunar gateway. 
Sec.~\ref{sec:conc} summarizes the paper.

\section{Preliminaries}
\label{sec:prelim}

$0_{n\times m}$ and $1_{n\times m}$ are matrices of zeros and ones in $\bbr^{n\times m}$ respectively,
$I_n$ is the $n$-dimensional identity matrix, 
$\Nint{a}{b}$ is the subset of natural numbers between (and including) $a,b\in \bbn$, $a \leq b$, $\llbracket a, b\rrbracket=\{x\in\bbr\mid a\leq x \leq b\}$, $e_i$ is the standard axis vectors of $\bbr^n$, and
${\| \cdot \|}_p$ is the $\ell_p$-norm of a vector. 
Let $M$ be a matrix and $M_1$ ($M_2$, resp.) be another matrix of the same height (width, resp.) as $M$. Then, $[M, M_1]$ ($[M; M_2]$, resp.) denotes the matrix obtained by concatenating $M$ and $M_1$ horizontally (concatenating $M$ and $M_2$ vertically, resp.). 
For a matrix $M\in\bbr^{m\times n}$ with full row rank, $M^\dagger=M^\top {\left({MM^\top}\right)}^{-1}$ denotes its
(right) pseudoinverse, and $x=M^\dagger v$ solves the system of linear equation $Mx=v$ for any vector $v\in\bbr^m$~\cite[Sec. 11.5]{boyd_introduction_2018}. 
Given $d\in\bbr^n$, $\operatorname{diag}(d)$ is a $n$-dimensional diagonal matrix with diagonal
entries $d_i$.

A set $\cals\subset\bbr^n$ is said to be \emph{symmetric} about any $c\in\bbr^n$, if $c+x\in\cals$ implies $c-x\in\cals$ for any
$x\in\bbr^n$. For any set $\cals\subseteq\bbr^n$, we denote its \emph{convex hull} and \emph{affine hull} by
$\convHull{\cals}$ and $\affHull{\cals}$ respectively. Recall that $\affHull{\cals}$ is an affine set such that $\cals\subseteq \cala \Rightarrow \affHull{\cals}\subseteq \cala$ for any affine set $\cala$. 
The \emph{affine dimension} of a set $\cals$ is the dimension of the subspace associated with $\affHull{\cals}$.
A \emph{full-dimensional} set in $\bbr^n$ is a non-empty set with an affine dimension of $n$. See~\cite[Sec.
2.1]{boyd2004convex} for more details.

\subsection{Set representations}
\label{sub:prelim_set_defn}

Let $\calc$ be a convex and compact polytope in $\bbr^n$. We consider two representations of $\calc$ --- \emph{H-Rep polytope} \eqref{eq:hrep_polytope} and \emph{constrained zonotope} \eqref{eq:constrained_zonotope},
\begin{subequations}
\begin{align}
    \calc&= \{x\mid H_Cx \leq k_C\}\label{eq:hrep_polytope},\\
    \calc&= \left\{G_C\dvec + c_C\ \middle\vert\  
                    {\|\dvec\|}_\infty \leq 1,\ 
                    A_C\dvec=b_C
    \right\}\label{eq:constrained_zonotope},
\end{align} \label{eq:set_representations}%
\end{subequations}
with $H_C\in\bbr^{\nineq{C}\times n}$, $k_C\in\bbr^\nineq{C}$, $G_C\in\bbr^{n\times \ndvec{C}}$, $c_C\in\bbr^{n}$, $A_C\in\bbr^{\nconst{C}\times
\ndvec{C}}$, and $b_C\in\bbr^{\nconst{C}}$.
Here, \eqref{eq:hrep_polytope} is the intersection of $\nineq{C}$ halfspaces and 
\eqref{eq:constrained_zonotope} is an affine transformation of $\calb_\infty(A_C, b_C)$,
$\calc = c_C + G_C\calb_\infty(A_C,b_C)$ with
\begin{align}
    \calb_\infty(A_C,b_C)\triangleq \{\dvec \mid {\|\dvec\|}_\infty \leq 1, A_C\dvec = b_C\}\label{eq:b_infty}.
\end{align}
In \eqref{eq:b_infty}, $\calb_\infty(A_C, b_C)$ is the intersection of a unit-hypercube in $\bbr^{\ndvec{C}}$ and 
$\nconst{C}$ linear equalities. 
By definition, 
$\calc\neq\emptyset$ if and only if $\calb_\infty(A_C, b_C)\neq\emptyset$.
Other respresentations of $\calc$, apart from \eqref{eq:set_representations}, include vertex representation (V-Rep polytope)~\cite{borrelli2017predictive} and AH-polytopes~\cite{sadraddini2019linear}.
We refer to \emph{unbounded} H-Rep polytopes as \emph{convex polyhedra}, and use $\calc=(G_C,c_C,A_C,b_C)$ to denote a polytope $\calc$ in constrained zonotope representation \eqref{eq:constrained_zonotope}.

The equivalence of the representations in \eqref{eq:set_representations} was established in~\cite[Thm. 1]{scott_constrained_2016}.  Additionally,~\cite[Thm. 1]{scott_constrained_2016} provides a systematic approach to convert H-Rep polytopes
\eqref{eq:hrep_polytope} to constrained zonotopes \eqref{eq:constrained_zonotope}. However, an exact conversion of \eqref{eq:constrained_zonotope} to \eqref{eq:hrep_polytope} is known to be computationally
demanding, with existing approaches applicable only for low-dimensional constrained zonotopes~\cite[Prop.
3]{scott_constrained_2016}.

The representations \eqref{eq:set_representations} are not unique.
For H-Rep polytopes, a canonical reduction (up to a permutation of rows) is available using linear programming~\cite{borrelli2017predictive,MPT3}.
However, no such canonical reduction is available for constrained zonotopes, to the best of our knowledge.
On the other hand, several exact and approximate techniques are available to reduce the \emph{representation complexity} of a given constrained zonotope~\cite{scott_constrained_2016,raghuraman2022set,kopetzki2017methods}.
\begin{definition}{\textsc{(Representation complexity)}}~\cite{scott_constrained_2016}
    Let $\dofo{C}\triangleq (N_C - M_C)/n$ be the \emph{degrees-of-freedom order} of a constrained zonotope $\calc=(G_C,c_C,A_C,b_C)$. Then, the representation complexity of $\calc$ is $\scrc(\calc)=(M_C, \dofo{C})$. \label{defn:repr}
\end{definition}

Zonotopes $\calz$, ellipsoids $\cale$, and convex unions of symmetric intervals $\cali$ are affine transformations of unit
balls defined using $\ell_\infty$-norms,
$\ell_2$-norms, and $\ell_1$-norms respectively.
Formally, we define $\calz,\cale,\cali\subset \bbr^n$ as follows,
\begin{subequations}
\begin{align}
    \calz&\triangleq \{G_Z\dvec + c_Z \mid {\|\dvec\|}_\infty \leq 1\}\label{eq:zonotope},\\
    \cale&\triangleq \{G_E\dvec + c_E \mid {\|\dvec\|}_2 \leq 1\}\label{eq:ellipsoid},\\
    \cali&\triangleq \{G_I\dvec + c_I \mid {\|\dvec\|}_1 \leq 1\}\label{eq:cui},
\end{align} \label{eq:set_representations_specific}%
\end{subequations}
with $G_Z\in\bbr^{n\times \ndvec{Z}}$, $G_E\in\bbr^{n\times n}$, $G_I\in\bbr^{n\times \ndvec{I}}$, and $c_Z,c_E,c_I\in\bbr^n$.
The sets $\calz,\cale,\cali$ are symmetric about $c_Z,c_E,c_I$.
We denote these sets using $(G, c)$.

For a convex and compact set $\cals\subset\bbr^n$, %
its support function $\rho:\bbr^n \to \bbr$ and support vector $\vartheta:\bbr^n \to \bbr^n$ are 
\begin{align}
    \vartheta_{\cals}(\nu)\triangleq \arg\sup\nolimits_{s\in\cals}\ \nu^\top s,\text{ and }\rho_{\cals}(\nu)\triangleq \nu^\top\vartheta_{\cals}(\nu)\label{eq:support}.
\end{align}
Using \emph{dual norms}~\cite[Sec. A.1.6]{boyd2004convex} and properties of support function~\cite[Prop. 2]{girard_efficient_2008}, we have closed-form
expressions for the support function of zonotopes $\calz$, ellipsoids $\cale$, and convex unions of symmetric intervals
$\cali$,
\begin{subequations}
\begin{align}
\rho_{\cali}(\nu) &= \nu^\top c_I + {\|G_I^\top \nu\|}_\infty\label{eq:support_cui},\\
\rho_{\cale}(\nu) &= \nu^\top c_E + {\|G_E^\top \nu\|}_2\label{eq:support_nuipsoid},\\
\rho_{\calz}(\nu) &= \nu^\top c_Z + {\|G_Z^\top \nu\|}_1\label{eq:support_zonotope}.
\end{align}\label{eq:support_sets}%
\end{subequations}%

\subsection{Set operations}
\label{sub:prelim_set_operations}

For any sets $\calc, \cals\subseteq\bbr^n$ and $\calw\subseteq\bbr^m$, and a matrix $R\in\bbr^{m\times n}$, we define the set operations (affine map, Minkowski sum $\oplus$, intersection with inverse affine map $\cap_R$, and Pontryagin difference $\ominus$):
\begin{subequations}
\begin{align}
    R\calc &\triangleq\{R u \mid u \in \calc\},\label{eq:affinemap}\\
    \calc \oplus \cals &\triangleq\{u + v \mid u \in \calc,\ v\in\cals\},\label{eq:msum}\\
    \calc \cap_R \calw &\triangleq\{u\in\calc \mid  Ru \in \calw\},\label{eq:intersection}\\
    \calc \ominus \cals &\triangleq\{u \mid  \forall v \in\cals, u + v \in \calc\}.\label{eq:pdiff}
\end{align}\label{eq:set_operations}%
\end{subequations}
Since $\calc\cap\cals=\calc\cap_{I_n}\cals$, \eqref{eq:intersection} also includes the standard intersection. For any
$x\in\bbr^n$, we use $\calc + x$ and $\calc - x$ to denote $\calc \oplus \{ x\}$ and $\calc \oplus \{-x\}$ respectively
for brevity. 

Constrained zonotopes admit closed-form
expressions for various set operations described in \eqref{eq:set_operations}. 
From~\cite{scott_constrained_2016,raghuraman2022set}, 
\begin{subequations}
    \begin{align}
    R\calc &= (RG_C,Rc_C,A_C,b_C),\label{eq:affinemap_CZ}\\   
    \calc \oplus \cals &= \left([G_C, G_S], c_C+c_S,[A_C, 0; 0, A_S],\right. \nonumber\\
    &\qquad\left.[b_C;b_S]\right)\label{eq:msum_CZ},\\
\calc \cap_R \calw &= \left([G_C, 0], c_C,[A_C, 0; 0, A_W;RG_C, -G_W],\right.\nonumber\\
&\qquad\left.[b_C;b_W;c_W - R c_C]\right)\label{eq:intersection_CZ},\\
\calc \cap \calh &= \left([G_C, 0], c_C,[A_C, 0; p^\top G_C, d_m/2],\right.\nonumber\\
&\qquad\left.[b_C;(q + p^\top c_C - \|p^\top G_C\|_1)/2]\right)\label{eq:intersection_CZ_H},
\end{align}\label{eq:set_operations_CZ}%
\end{subequations}
where $\calh=\{x \mid p^\top x \leq
q\}\subset\bbr^n$ is an halfspace, and  
\eqref{eq:intersection_CZ_H} also enables an exact computation of the intersection of a constrained zonotope and a
convex polyhedron. %

To the best of our knowledge, the Pontryagin difference \eqref{eq:pdiff} involving a constrained zonotopic minuend does not have a closed-form expression, similar to \eqref{eq:set_operations_CZ}~\cite{raghuraman2022set,yang_efficient_2022}.
In fact, given a constrained zonotope $\calc$ and a zonotope $\calz$, it is impossible to find a polynomial-size constrained zonotope $\calc\ominus\calz$ in polynomial-time, unless P=NP~\cite[Prop. 1]{yang_efficient_2022}.  
On the other hand,  
given a H-Rep polytope $\calp=\{x\mid H_P x \leq k_P\}$
and a convex and compact
subtrahend $\cals$,
a H-Rep polytope $\calp\ominus\cals$ is available in closed-form, 
\begingroup
    \makeatletter\def\f@size{9.5}\check@mathfonts
\begin{align}
     \calp\ominus\cals &=\{x \mid H_Px\leq k_P - [\rho_{\cals}(h_1);\ \ldots; \rho_{\cals}(h_{\nineq{P}})]\}\label{eq:pdiff_polytope},
\end{align}
\endgroup
with $H_P=[h_1^\top;h_2^\top;\ldots;h_{\nineq{C}}^\top]$~\cite[Thm. 2.3]{kolmanovsky1998theory}. 
A direct use of \eqref{eq:pdiff_polytope} to compute a constrained zonotope $\calc\ominus\cals$ for a constrained zonotope minuend $\calc$ is prevented by the difficulty in converting \eqref{eq:constrained_zonotope} to \eqref{eq:hrep_polytope}.

Recently,~\cite{yang_efficient_2022} proposed a two-stage approach to inner-approximate $\calc\ominus\cals$
between a constrained zonotopic minuend $\calc$ and a zonotopic subtrahend $\cals$ with
$G_S=\left[{g_{S}^{(1)},\ldots,g_{S}^{(\ndvec{S})}}\right]$.
The two-stage approach first solves a linear program with $2\ndvec{C}\ndvec{S}$ variables,
\begin{align}
    \hspace*{-0.75em}\begin{array}{cl}
    {\text{minimize}} &\ \  \operatorname{SumAbs}(\Gamma)\\
    \text{s.\ t.} &\ \ [G_C;A_C]\Gamma=[G_S;0_{\nconst{C}\times \ndvec{S}}],\\
    \forall i\in\Nint{1}{\ndvec{C}}, &\ \ \sum_{j=1}^{\ndvec{S}} |\Gamma_{ij}|\leq 1,\\
    \end{array}\label{eq:yang_pdiff_cz_minus_v_lp}%
\end{align}
where $\operatorname{SumAbs}(\Gamma)\triangleq\sum_{i=1}^{\ndvec{C}}\sum_{j=1}^{\ndvec{S}} |\Gamma_{ij}|$. Let the optimal solution to 
\eqref{eq:yang_pdiff_cz_minus_v_lp} be $\Gamma^\ast$, if it exists, and the second stage  defines a diagonal matrix $D\in\bbr^{\ndvec{C}\times\ndvec{C}}$ with 
\begin{align}
D_{ii}=1-{\|e_i^\top \Gamma^\ast\|}_1
=1 -
\sum_{j=1}^{\ndvec{S}}|\Gamma^\ast_{ij}|\label{eq:yang_pdiff_cz_minus_v_d_ii}
\end{align}
for each $i\in\Nint{1}{\ndvec{C}}$. Then, a constrained zonotopic  inner-approximation of $\calc\ominus\cals$ is  
\begin{align}
    \calm^-\triangleq (G_CD,c_C - c_S, A_CD,b_C)\subseteq 
    \calc\ominus\cals\label{eq:m_minus_yang}.
\end{align}
Here, $\scrc(\calm^-)=\scrc(\calc)$, which is desirable when using
\eqref{eq:m_minus_yang} in
set recursions involving the Pontryagin difference. However, solving \eqref{eq:yang_pdiff_cz_minus_v_lp} repeatedly can become computationally
expensive for large $\ndvec{C},\ndvec{S}$.
Also, it is unclear if such an approach extends to subtrahends beyond zonotopes,
since it uses polytopic containment results from~\cite{sadraddini2019linear}. We address these
shortcomings in this paper.
\begin{problem}\label{prob_st:pdiff}
    Let $\calc$ be a constrained zonotope and $\cals$ be a convex and compact set that is symmetric about 
    $c_S\in\bbr^n$. Characterize constrained zonotopes
    $\calm^-,\calm^+\subset \bbr^n$ with 
    \begin{align}
        \calm^-\subseteq\calm\triangleq\calc\ominus\cals\subseteq\calm^+,\label{eq:prob_st_m_minus_plus}
    \end{align}
    where $\calm^-$ is given in closed-form and
    $\scrc(\calm^-)=\scrc(\calc)$. Also, provide sufficient conditions for $\calm^-=\calm=\calm^+$.
\end{problem}
Our approach to address Prob.~\ref{prob_st:pdiff} follows a similar two-stage approach to that in~\cite{yang_efficient_2022} to obtain \eqref{eq:m_minus_yang}. However, unlike~\cite{yang_efficient_2022}, we provide closed-form expressions for $\Gamma$
and $D$ used in \eqref{eq:yang_pdiff_cz_minus_v_d_ii} and \eqref{eq:m_minus_yang} for a broader
class of subtrahends that includes ellipsoids and convex unions of symmetric 
intervals. 
We also propose a constrained zonotopic outer-approximation to $\calm$, and provide sufficient
conditions under which these approximations are exact. The proposed outer-approximation relies on solving the following
problem.
\begin{problem}\label{prob_st:outer_approx}
    Given a constrained zonotope $\calc$, design an algorithm to compute an outer-approximating convex
    polyhedron $\calp$ (i.e, $\calp\supseteq\calc$) in the form of \eqref{eq:hrep_polytope} that has at most
    $2\ndvec{C}$ linear constraints.  Also, provide sufficient
    conditions under which $\calc=\calp$.
\end{problem}
We will also briefly discuss the relationship between the proposed approach and the two-stage approach in~\cite{yang_efficient_2022}. 

\subsection{Robust controllable set}

Consider a linear time-varying system,
\begin{align}
    x_{t+1} = A_t x_t + B_t u_t + F_t w_t\label{eq:ltv_dyn},
\end{align}
with state $x_t\in\bbr^n$, input $u_t\in\calu_t\subset\bbr^m$, disturbance $w_t\in\calw_t\subset\bbr^p$, and appropriately
defined time-varying matrices $A_t$, $B_t$, and $F_t$. 
\begin{definition}{\textsc{($T$-Step RC set)}~\cite[Defn. 10.18]{borrelli2017predictive}}\label{defn:RC_set}
Given \eqref{eq:ltv_dyn}, a set of (possibly time-varying) state constraints ${\{\calx_t\}}_{t=0}^{T-1}$
with $\calx_t\subseteq\bbr^n$ for each $t$, and a goal set $\calg\subset\bbr^n$, the $T$-step robust controllable
(RC) set is
    \begin{align}
    \calk = \left\{x_0\in\calx_0\ \middle\vert\ \begin{array}{c}
    \forall t\in\Nint{0}{T-1},\ \exists u_t\in\calu_t,\ \forall w_t\in\calw_t,\\
     x_{t+1} = A_t x_t + B_t u_t + F_t w_t,\\
     x_{t}\in\calx_{t},\ x_T\in\calg
    \end{array}\right\}.\nonumber%
\end{align}
\end{definition}
Informally, the $T$-step RC set $\calk \subset \bbr^n$ is the set of initial states that, despite the additive disturbance $w_t\in\calw_t$, can be steered using $u_t\in\calu_t$ to reach the
goal set $\calg$ at time step $T$, while staying within the state constraints $\calx_t$ at
all intermediate time steps $t\in\Nint{0}{T-1}$. The $T$-step RC set is also known as robust
reachability of target tube~\cite{bertsekas_1971}, backward
reachable set~\cite{yang_efficient_2022}, or backward reach-avoid set~\cite{gleason2021lagrangian}.
The following set recursion yields $\calk=\calk_0$,
\begingroup
\makeatletter\def\f@size{8.5}\check@mathfonts
\begin{subequations}
\begin{align}
    \calk_t &= \operatorname{Pre}(\calk_{t+1}) \cap \calx_{t},\qquad\quad \forall
    t\in\Nint{0}{T-1},\\
    \operatorname{Pre}(\calk_{t+1})&\triangleq \{x\mid A_tx \in (\calk_{t+1} \ominus F_t\calw_t) \oplus (-B_t \calu_t)\},\label{eq:pre_defn}
\end{align}\label{eq:set_recursion}%
\end{subequations}
\endgroup
with $\calk_T\triangleq\calg$.
Observe that \eqref{eq:set_recursion} uses all set operations in \eqref{eq:set_operations}. For H-Rep/V-Rep polytopes  
$\calx,\calg,\calu$, and $\calw$, we can compute RC sets using existing results~\cite{MPT3,gleason2021lagrangian,borrelli2017predictive}.
However, when using H-Rep/V-Rep polytopes for high-dimensional systems \eqref{eq:ltv_dyn} or over long
horizons $T$, we face numerical
challenges when implementing \eqref{eq:set_recursion}, since it requires a combination of Minkowski sum and intersection operations. See~\cite{borrelli2017predictive,Blanchini2015,althoff2021set,kerrigan2000robust} for a detailed discussion.

On the other hand, constrained zonotopes provide a tractable, exact solution when $\calw=\emptyset$ using
\eqref{eq:set_operations_CZ}~\cite{scott_constrained_2016,raghuraman2022set,yang_efficient_2022}. Also, for a zonotopic
$\calw$,~\cite{yang_efficient_2022} uses the two-stage approach reviewed in Sec.~\ref{sub:prelim_set_operations} to
propose an inner-approximation  of $\calk$. 
In this paper, we use our solution to Prob.~\ref{prob_st:pdiff} to inner-approximate $\calk$ for a broader class of
 sets $\calw$.
\begin{problem}\label{prob_st:RC_set}
    Compute a constrained zonotope $\calk^-$ that inner-approximates $\calk$ defined in Defn.~\ref{defn:RC_set}, where for every $t\in\Nint{0}{T-1}$, $\calu_t$ and $\calg$ are polytopes, $\calw_t$ are symmetric,
    convex, and compact sets, and either 1) 
    $\calx_t$ are convex polyhedra and $A_t$ are invertible, or 2) $\calx_t$ are polytopes.
    Additionally, characterize the representation complexity of $\calk^-$ in both cases.
\end{problem}
In both settings considered, the exact RC sets are known to be polytopes~\cite{borrelli2017predictive}, and hence are representable via constrained zonotopes.
We also use the proposed outer-approximation to the Pontryagin
difference to outer-approximate $\calk_t$ when using \eqref{eq:set_recursion}. %

\section{Inner-approximation of Pontryagin difference}
\label{sec:inner_approx}

Given a constrained zonotopic minuend $\calc\subset \bbr^n$ and a symmetric, convex, and compact subtrahend
$\cals\subset\bbr^n$, we characterize a constrained zonotope $\calm^-\subset \bbr^n$ where
\begin{align}
    \calm^-\subseteq\calm\triangleq\calc\ominus\cals.\label{eq:pdiff_inner_approx}
\end{align}
Specifically, we provide closed-form expressions for a diagonal matrix $D\in\bbr^{\ndvec{C}\times \ndvec{C}}$ using the properties of $\calc$ and $\cals$, and define $\calm^-=
(G_{M^-},c_{M^-},A_{M^-},b_{M^-})$ with
\begin{subequations}
    \begin{align}
        G_{M^-} &= G_C D,&&& c_{M^-} &= c_C - c_S,\\
        A_{M^-} &= A_C D,&&& b_{M^-} &= b_C,
    \end{align}\label{eq:pdiff_inner_approx_parameters}%
\end{subequations}
for a specific $c_S\in\cals$, similarly to \eqref{eq:m_minus_yang}.
The representation complexity of $\calm^-$ is $\scrc(\calm^-)=(\nconst{C}, \dofo{C})=\scrc(\calc)$.

In this section, we propose two constrained zonotope representations that allow us to address Prob.~\ref{prob_st:pdiff}, and relate our approach with the two-stage approach
in~\cite{yang_efficient_2022} for zonotopic subtrahends.
Table~\ref{tab:theory} summarizes the computation of $D$ for various types of sets.%
\begin{table*}
    \centering
    \adjustbox{width=1\linewidth}{\begin{tabular}{|p{22em}|p{24.5em}|p{3.5em}|}
    \hline
    \hfill Subtrahend $\cals$ & Definition of the diagonal matrix $D=\operatorname{diag}_{i\in\Nint{1}{\ndvec{C}}}(D_{ii})$ & Result \\\hline\hline
    \hfill Convex and compact $\cals$, symmetric about $c_S\in\bbr^n$ & $D_{ii} = 1 -
    \rho_{\cals_0}\left({e_i^\top
{[G_C;A_C]^\dagger[I_n;0_{\nconst{C}\times n}]}}\right)$ for $\cals_0\triangleq \cals - c_S$ &
Prop.~\ref{prop:pdiff_general_l2}\\\hline
    \hfill \ Zonotope $\cals$ \eqref{eq:zonotope} 
    & $D_{ii} = 1 - {\left\|{e_i^\top [G_C;A_C]^\dagger[G_S;0_{\nconst{C}\times
\ndvec{S}}]}\right\|}_1$ & \\\cline{1-2}
    \hfill Ellipsoid $\cals$ \eqref{eq:ellipsoid} 
    & $D_{ii} = 1 - {\left\|{e_i^\top [G_C;A_C]^\dagger[G_S;0_{\nconst{C}\times
 n}]}\right\|}_2$ &  Corr.~\ref{corr:specific}\\\cline{1-2}
    \raggedleft{Convex union of symmetric intervals $\cals$ \eqref{eq:cui}}  
    & $D_{ii} = 1 - {\left\|{e_i^\top [G_C;A_C]^\dagger[G_S;0_{\nconst{C}\times
\ndvec{S}}]}\right\|}_\infty$
    & \\\hline
    \end{tabular}}%
    \caption{Summary of the closed-form expressions for the diagonal matrix $D=\operatorname{diag}_{i\in\Nint{1}{\ndvec{C}}}(D_{ii})$ for various subtrahends $\cals$  that is symmetric about $c_S\in\bbr^n$, and the corresponding result in the paper. Given a full-dimensional constrained zonotopic minuend $\calc=(G_C,c_C,A_C,b_C)$, we propose a constrained zonotope $\calm^-=(G_C D, c_C - c_S, A_C D, b_C)$ that satisfies $\calm^-\subseteq\calm=\calc\ominus\cals$.}
    \label{tab:theory}
\end{table*}

\subsection{Full-dimensional constrained zonotopes}
\label{sub:aux}

\begin{figure}[t]
    \centering
    \includegraphics[width=0.5\linewidth]{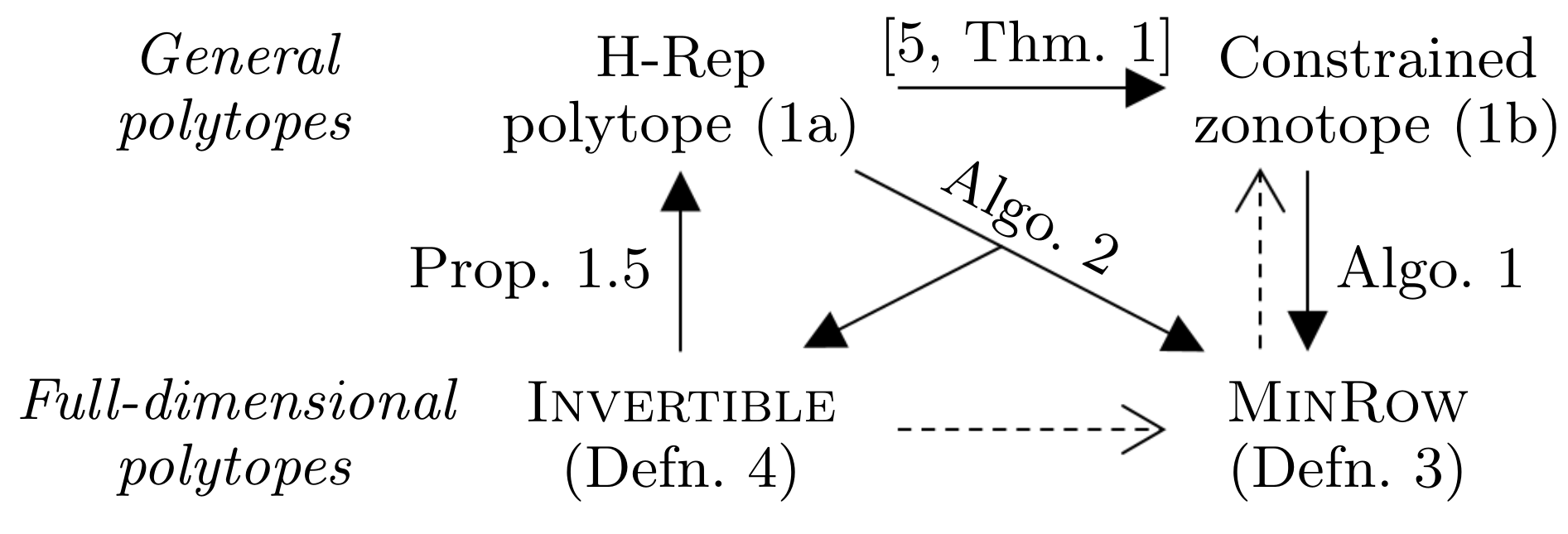}
    \caption{Relationship between various representations discussed in the paper. For any pair of representations $R_1, R_2$, a dashed arrow from $R_1$ to $R_2$ shows that $R_1$ is also $R_2$.}\label{fig:transforms}
\end{figure}

We study two constrained zonotope representations that are guaranteed to exist for full-dimensional polytopes.
\begin{definition}{\textsc{(MinRow Representation)}}
A constrained zonotope $\calc=(G_C,c_C,A_C,b_C)$ is a \textsc{MinRow} representation when $[G_C;A_C]$ has 
full row rank.\label{defn:MinRow}
\end{definition}
\begin{definition}{\textsc{(Invertible representation)}}
    A \textsc{MinRow} representation $\calc$ is an \textsc{Invertible} representation when $n+\nconst{C}=\ndvec{C}$, i.e., $\dofo{\calc}=1$.\label{defn:invertible}
\end{definition}
By definitions, the matrix $[G_C;A_C]\in\bbr^{(n+\nconst{C})\times \ndvec{C}}$ has a well-defined right pseudoinverse ${[G_C;A_C]}^\dagger$  for a \textsc{MinRow} representation, and $[G_C;A_C]$ is invertible for an \textsc{Invertible} representation.
Fig.~\ref{fig:transforms} illustrates the relationships between the representations based on Prop.~\ref{prop:Indep}.

\begin{algorithm}[t]
\caption{Compute \textsc{MinRow} representation}\label{algo:min_row}
\begin{algorithmic}[1]
    \Require Constrained zonotope $(G_C,c_C,A_C,b_C)$ that represents a full-dimensional polytope

    \Ensure \textsc{MinRow} representation  $(G_C',c_C',A_C',b_C')$
    
    \State Set $G_C'\gets G_C$ and $c_C'\gets c_C$.

    \State Get $A_C',b_C'$ from rows of $[A_C,b_C]$ where $[A_C',b_C']$ has full row rank and $\operatorname{rank}([A_C',b_C'])=\operatorname{rank}([A_C,b_C])$.\label{step:ACprime}
\end{algorithmic}
\end{algorithm}
\begin{algorithm}[t]
\caption{Compute \textsc{\textsc{Invertible}} representation}\label{algo:invertible}
\begin{algorithmic}[1]
    \Require H-Rep polytope $\calc=\{x \mid H_Cx\leq k_C\}\subset\bbr^n$ that represents a full-dimensional polytope

    \Ensure \textsc{Invertible} representation  $(G_C,c_C,A_C,b_C)$
    
    \State Define a zonotope $\calz=(G_Z, c_Z)$ \eqref{eq:zonotope} with $G_Z\triangleq{\operatorname{diag}\left(\frac{u - l}{2}\right)\in\bbr^{n\times n},c_Z\triangleq\frac{u + l}{2}}\in\bbr^n$ for $l,u\in\bbr^n$,
    \begin{align}
       l_i=-\rho_{\calc}(-e_i),\ \ u_i=\rho_{\calc}(e_i),\quad \forall i\in\Nint{1}{n}\label{eq:invertible_zonotope}.
    \end{align}
    \State Define $\sigma\in\bbr^{\nineq{C}}$ with $\sigma_i=-\rho_{\calz}(-h_i)$ for each $i\in\Nint{1}{\nineq{C}}$ and $H_C=[h_1^\top;h_2^\top;\ldots;h_{\nineq{C}}^\top]$.
    \State Compute the \textsc{Invertible} representation for $\calc$ with
\begingroup
    \makeatletter\def\f@size{9}\check@mathfonts
    \begin{align}
        G_{C} &= [G_Z, 0_{n\times \nineq{C}}],& c_C &= c_Z,\nonumber\\
        A_{C} &= \left[{H_C G_Z, \operatorname{diag}\left(\frac{\sigma - k}{2}\right)}\right],& b_C &= \frac{\sigma + k}{2} - H_Cc_Z. \nonumber
    \end{align}%
\endgroup
\end{algorithmic}
\end{algorithm}

\begin{proposition}\label{prop:Indep}
    The following statements hold:\\
    1) Every representation $(G_C,c_C,A_C,b_C)$ of a full-dimensional polytope $\calc$ satisfies $\operatorname{rank}(G_C)=n\leq \ndvec{C}$. \\
    2) A polytope is full-dimensional if and only if the polytope is non-empty and it has a \textsc{MinRow} representation.\\
    3) Algo.~\ref{algo:min_row} generates a \textsc{MinRow} representation from any full-dimensional constrained zonotope \eqref{eq:constrained_zonotope}.\\
    4) Algo.~\ref{algo:invertible} generates an \textsc{Invertible} representation from any full-dimensional H-Rep polytope \eqref{eq:hrep_polytope}.\\
    5) Given an \textsc{Invertible} representation $(G_C,c_C,A_C,b_C)$, the corresponding H-Rep polytope \eqref{eq:hrep_polytope} is given by,
    \begin{subequations}
    \begin{align}
        \hspace*{-2em}H_C &= \left[I_{\ndvec{C}}; -I_{\ndvec{C}}\right]{\left[G_C;A_C\right]}^{-1}\left[I_{n}; 0_{\nconst{C}\times n}\right],\\
        k_C & = 1_{2\ndvec{C}\times 1} - 
        \left[I_{\ndvec{C}}; -I_{\ndvec{C}}\right]{\left[G_C;A_C\right]}^{-1}\left[-c;b\right].
    \end{align}\label{eq:hrep_from_cz}
\end{subequations}
\end{proposition}
See Sec.~\ref{app:proofs_aux_prop_indep} for the proof of Prop.~\ref{prop:Indep}.
Step~\ref{step:ACprime} in Algo.~\ref{algo:min_row} removes redundant equalities in $\{\dvec\mid A_C\dvec=b_C\}$. See \texttt{mpt\_minAffineRep} in MPT3~\cite{MPT3} for an implementation.

By Prop.~\ref{prop:Indep}.4, an \textsc{Invertible} representation also exists for every full-dimensional constrained zonotope. However, it is unclear if we can efficiently compute such a representation from an arbitrary full-dimensional constrained zonotope, as done in Algo.~\ref{algo:min_row}.

\begin{remark}
The rest of the paper assumes the use of a \textsc{MinRow} representation for any full-dimensional constrained zonotope, obtained using either Algo.~\ref{algo:min_row} or~\ref{algo:invertible}.\label{rem:min_row}
\end{remark}

\subsection{Inner-approximation of the Pontryagin difference}
\label{sub:general}

Thm.~\ref{thm:pdiff_general} tackles the inner-approximation part of  Prob.~\ref{prob_st:pdiff}.

\begin{theorem}\label{thm:pdiff_general}
    Given a non-empty constrained zonotopic minuend $\calc=(G_C,c_C,A_C,b_C)\subset \bbr^n$ and a convex and compact subtrahend
    $\cals\subset \bbr^n$ that is symmetric about any $c_S\in\bbr^n$. Let 
    $\Gamma:\bbr^{\ndvec{C} \times n}$ solve 
    \begin{align}
        [G_C;A_C] \Gamma = [I_n;0_{\nconst{C}\times n}],\label{eq:Gamma_linear_equations}
    \end{align}
    and $D\in\bbr^{\ndvec{C} \times \ndvec{C}}$ be a diagonal matrix with 
    \begin{align}
        D_{ii}&\triangleq 1 - \rho_{\cals_0}\left({ e_i^\top \Gamma}\right)\label{eq:diag_matrix_general},
    \end{align}
    for each $i\in\Nint{1}{\ndvec{C}}$, where $\cals_0\triangleq \cals -
    c_S$.  
    Then, the constrained zonotope $\calm^-$ defined using \eqref{eq:pdiff_inner_approx_parameters} and $D$ in \eqref{eq:diag_matrix_general} satisfies
    \eqref{eq:pdiff_inner_approx}, provided $D_{ii} \geq 0$ for every $i\in\Nint{1}{\ndvec{C}}$.
\end{theorem}
\begin{proof}
    Given $\cals_0\subset\bbr^n$, define $\calv_0 = \Gamma\cals_0\subset\bbr^{\ndvec{C}}$. 
    By definition of $\Gamma$, $\calv_0$ satisfies 1)
    $G_C\calv_0=\cals_0=\cals - c_S$, and 2) $A_C\calv_0=\{0_{\nconst{C}\times 1}\}$. 
    
    Next, define $\calc_{L} \triangleq\calb_\infty(A_C,b_C)\ominus \calv_0\subset
    \bbr^{\ndvec{C}}$, and
    define 
    \begin{align}
        \calm^-\triangleq c_{M^-} + G_C\calc_L.\label{eq:m_minus_defn_general}
\end{align}
with $c_{M^-}$ in \eqref{eq:pdiff_inner_approx_parameters}. 
Recall that for any set $\cala,\calb\subset \bbr^n$ and a matrix $M\in\bbr^{m\times n}$ with $m<n$, $M(\cala\ominus\calb)\subseteq(M\cala)\ominus(M\calb)$~\cite[Lem. 2(ii)]{yang_efficient_2022}.
From~\cite[Thm. 2.1.iv]{kolmanovsky1998theory}, 
\begin{align}
     \calm^-&\subseteq c_{M^-} + (G_C\calb_\infty(A_C,b_C)\ominus
(G_C\calv_0)) \nonumber\\
                     &=c_{M^-} + ((\calc - c_C) \ominus (\cals - c_S))\nonumber\\
                     &=(c_{M^-} - (c_C - c_S)) + (\calc \ominus \cals) =\calc \ominus \cals.\label{eq:subset_relation_general}
\end{align}
Next, we simplify $\calc_L$ as follows, %
\begingroup
    \makeatletter\def\f@size{9.5}\check@mathfonts
    \begin{align}
    \calc_L&=\{\dvec \mid \forall v\in\calv_0,\ \dvec + v\in\calb_\infty(A_C,b_C)\}
    \label{eq:general_interim_1}\\
                 &=\{\dvec\mid 
            \forall v\in\calv_0,\ 
        A_C(\dvec+v)=b_C,\ {\|\dvec+v\|}_\infty\leq 1
    \} \label{eq:general_interim_2}\\
                 &=\{\dvec\mid 
            A_C\dvec=b_C,\ 
            \forall v\in\calv_0,\ 
        {\|\dvec+v\|}_\infty\leq 1\} \label{eq:general_interim_3}\\
                 &=\{\dvec\mid  
                 A_C\dvec=b_C\}\cap\left(\{\dvec\mid 
        {\|\dvec\|}_\infty\leq 1\}\ominus\calv_0\right) \label{eq:general_interim_4}\\
                 &=\left\{\dvec\ \middle\vert
                     \begin{array}{c}
                         A_C\dvec=b_C,\ \forall i\in\Nint{1}{\ndvec{C}},\ \forall \delta\in\{-1,1\},\\
                     \delta e_i^\top \dvec \leq 1 - \rho_{\calv_0}\left(\delta e_i\right)
             \end{array}\right\} \label{eq:general_interim_6}\\
                 &=\left\{\dvec\ \middle\vert\ 
                     \begin{array}{c}
                         A_C\dvec=b_C,\ \forall i\in\Nint{1}{\ndvec{C}},\\
                     | e_i^\top \dvec | \leq 1 - \rho_{\cals_0}\left(\Gamma^\top
                     e_i\right)
             \end{array}\right\} \label{eq:general_interim_7}\\
                 &=\left\{\dvec\ \middle\vert
                         A_C\dvec=b_C,\ \forall i\in\Nint{1}{\ndvec{C}},\ \left|e_i^\top \dvec \right| \leq D_{ii}
             \right\} \label{eq:general_interim_8}.
\end{align}
\endgroup
Here, 
\eqref{eq:general_interim_1} follows from \eqref{eq:pdiff} and the definition of $\calc_L$, 
\eqref{eq:general_interim_2} follows from \eqref{eq:b_infty}, 
\eqref{eq:general_interim_3} follows from the choice of $\calv_0$,
\eqref{eq:general_interim_4} follows from \eqref{eq:pdiff},
\eqref{eq:general_interim_6} follows from \eqref{eq:pdiff_polytope},
\eqref{eq:general_interim_7} follows from \eqref{eq:support} and from
$\rho_{\calv_0}(\nu)=\rho_{\cals_0}(\Gamma^\top \nu)$ for all $\nu\in\bbr^{\ndvec{C}}$~\cite[Prop. 2]{girard_efficient_2008}, and from symmetry of $\cals_0$ about the origin implying $\rho_{\cals_0}(-\mu)=\rho_{\cals_0}(\mu)$ for all $\mu\in\bbr^n$ by
\eqref{eq:support},
and \eqref{eq:general_interim_8} follows from 
the definition of $D_{ii}$ \eqref{eq:diag_matrix_general}.

If $D_{ii}< 0$ for any $i\in\Nint{1}{\ndvec{C}}$, then $\calc_L=\emptyset$ by \eqref{eq:general_interim_8}, and $\calm^-$ defined in  \eqref{eq:m_minus_defn_general} is also empty. On the other hand, when $D_{ii}\geq 0$ for all $i\in\Nint{1}{\ndvec{C}}$, 
\eqref{eq:general_interim_8} may be expressed as a scaled version of $\calb_\infty$,
\begin{align}
 \calc_L&=\{D\dvec\mid 
             A_C D\dvec=b_C,\ {\|\dvec\|}_\infty \leq 1\}=D\calb_\infty(A_CD,b_C).\nonumber%
\end{align}
Thus, $\calm^-\subseteq \calc\ominus \cals$ in \eqref{eq:pdiff_inner_approx} using \eqref{eq:m_minus_defn_general} and
\eqref{eq:subset_relation_general}, provided $D_{ii}\geq 0$ for all $i\in\Nint{1}{\ndvec{C}}$. 
\end{proof}

Thm.~\ref{thm:pdiff_general} inner-approximates the Pontryagin difference \eqref{eq:pdiff} using the support function
of the subtrahend $\cals$. 
The condition $D_{ii}\geq 0$
for all $i\in\Nint{1}{\ndvec{C}}$ in Thm.~\ref{thm:pdiff_general} may be viewed as requiring $\Gamma$ to additionally
satisfy $\Gamma\cals_0\subseteq \{\dvec\mid
\|\dvec\|_\infty \leq 1\}$. Such a choice ensures $\rho_{\cals_0}\left(\Gamma^\top
e_i\right)\leq 1$, and thereby, $D_{ii}\geq 0$~\cite[Ex. 3.35(d)]{boyd2004convex}.

\begin{proposition}\label{prop:pdiff_general_l2}
    For a full-dimensional, constrained zonotopic minuend $\calc$ and a convex and compact subtrahend
    $\cals\subset \bbr^n$ that is symmetric about any $c_S\in\bbr^n$, the constrained zonotope $\calm^-$ defined in Thm.~\ref{thm:pdiff_general} with
    $\Gamma=[G_C;A_C]^\dagger[I_n;0_{\nconst{C}\times n}]$ inner-approximates $\calc\ominus\cals$.
\end{proposition}
\begin{proof}
    Follows from Rem.~\ref{rem:min_row} and Thm.~\ref{thm:pdiff_general}.
\end{proof}
Prop.~\ref{prop:pdiff_general_l2} addresses Prob.~\ref{prob_st:pdiff}.
Additionally, we use the following two observations to illustrate that the assumption of full-dimensionality in Prop.~\ref{prop:pdiff_general_l2} is not restrictive. First, 
$\calc\ominus\cals=\emptyset$, whenever the affine dimension of $\cals$ is strictly greater than the affine dimension of
$\calc$ by \eqref{eq:pdiff}, since $x+\cals$ can not be contained in $\calc$ for any $x\in\bbr^n$ in such a scenario.
Second, consider the case where both minuend and subtrahend are not full-dimensional and there exists a matrix $T\in\bbr^{n\times n'}$ with $n>n'$ and $\text{rank}(T)=n'$ such that $\calc=T\calc'$ and $\cals=T\cals'$ for any $\calc',\cals'\subset\bbr^{n'}$.
Then, by~\cite[Thm. 2.1.viii]{kolmanovsky1998theory}, 
\begin{align}
    \calc\ominus\cals=(T\calc')\ominus(T\cals')=T(\calc'\ominus\cals')\label{eq:low_dim}.
\end{align}
In such a case, it suffices to assume that $\calc'$ is full-dimensional in order to obtain a constrained zonotope inner-approximation of $\calc\ominus\cals$.

\begin{corollary}[\textsc{Special cases}]\label{corr:specific}
    For a full-dimensional constrained zonotopic minuend $\calc$, 
    $D$ in \eqref{eq:diag_matrix_general} has closed-form expressions for $D_{ii}$ for each $i\in\Nint{1}{\ndvec{C}}$:\\
    1) when $\cals$ is a zonotope \eqref{eq:zonotope}, 
    \begin{align}
        D_{ii} = 1 - {\left\|{e_i^\top [G_C;A_C]^\dagger[G_S;0_{\nconst{C}\times
\ndvec{S}}]}\right\|}_1
        \label{eq:D_ii_zonotope},
    \end{align}
    2) when $\cals$ is an ellipsoid \eqref{eq:ellipsoid}, 
    \begin{align}
        D_{ii} = 1 - {\left\|{e_i^\top [G_C;A_C]^\dagger[G_S;0_{\nconst{C}\times
\ndvec{S}}]}\right\|}_2
        \label{eq:D_ii_ellipsoid},
    \end{align}
    3) when $\cals$ is a convex union of symmetric intervals \eqref{eq:cui}, 
    \begin{align}
        D_{ii} = 1 - {\left\|{e_i^\top [G_C;A_C]^\dagger[G_S;0_{\nconst{C}\times
\ndvec{S}}]}\right\|}_\infty
        \label{eq:D_ii_cui}.
    \end{align}
\end{corollary}
\begin{proof}
    Follows from \eqref{eq:support_sets} and Prop.~\ref{prop:pdiff_general_l2}.
\end{proof}
Corr.~\ref{corr:specific} lists some broad classes of subtrahends that admits closed-form expressions for 
the inner-approximation $\calm^-$ of the Pontryagin difference $\calc\ominus\cals$.

\begin{algorithm}[t]
\caption{Inner-approximation of 
$\calc\ominus\cals$}\label{algo:inner_approx}
\begin{algorithmic}[1]
    \Require Minuend $\calc=(G_C,c_C,A_C,b_C)$ that is full-dimensional and subtrahend $\cals$ that is convex, compact, and symmetric about any $c_S\in\bbr^n$ 

    \Ensure Constrained zonotope $\calm^-\subseteq \calc\ominus \cals$
    
    \State $\cals_0\gets \cals - c_S$

    \State $\Gamma\gets[G_C;A_C]^\dagger[I_n;0_{\nconst{C}\times n}]$\label{step:Gamma}

    \State $D\gets\operatorname{diag}([1 -
    \rho_{\cals_0}(e_1^\top \Gamma);\ldots;1 -
    \rho_{\cals_0}(e_{\ndvec{C}}^\top \Gamma)])$ \label{step:d_ii_compute}

    \State 
\begingroup
    \makeatletter\def\f@size{8.5}\check@mathfonts
 $\calm^-\gets
    \begin{cases}
        \begin{array}{ll}
            (G_C D, c_C - c_S, A_C D, b_C), &\quad \min\limits_{i\in\Nint{1}{\ndvec{C}}} D_{ii}\geq 0,\\
            \emptyset, &\quad\text{otherwise}
        \end{array}.
    \end{cases}$\label{step:check}
\endgroup
  \end{algorithmic}
\end{algorithm}
Algo.~\ref{algo:inner_approx} summarizes the procedure described in Prop.~\ref{prop:pdiff_general_l2} to inner-approximate $\calc\ominus\cals$ when
the minuend $\calc$ is full-dimensional.
For subtrahends that are ellipsoids, zonotopes, or convex unions of symmetric 
intervals, Step~\ref{step:d_ii_compute} in Algo.~\ref{algo:inner_approx} is available in closed form (see Corr.~\ref{corr:specific}), and Algo.~\ref{algo:inner_approx} is optimization-free, i.e., it does not require convex optimization solvers.

\begin{remark}
We can also use a ``template'' constrained zonotope $\calc^-\subseteq\calc$ when available and 
compute $\calc^-\ominus\cals\subseteq\calc\ominus\cals$ using Thm.~\ref{thm:pdiff_general}. Such an approach may help in overcoming conservativeness in certain cases, e.g., $\calm^-$ described in Thm.~\ref{thm:pdiff_general} is a zonotope for zonotopes $\calc,\cals$, even when $\calc\ominus\cals$ is known to be a constrained zonotope~\cite{yang_efficient_2022,raghuraman2022set}.
\end{remark}

\subsection{Relation of Algo.~\ref{algo:inner_approx} with existing two-stage approach for zonotopic subtrahend}
\label{sub:comparison}

For a zonotopic subtrahend $\cals$, the optimization-free Algo.~\ref{algo:inner_approx} is closely related to the optimization-based two-stage approach~\cite{yang_efficient_2022} (see Sec.~\ref{sub:prelim_set_operations}).

Let $\cals = G_S \cals_0 + c_S$ where $\cals_0$ is the unit $\ell_\infty$-norm
ball \eqref{eq:zonotope}.
Consider the \emph{matrix least norm} problem~\cite[Ex.
16.2]{boyd_introduction_2018}, which is similar to \eqref{eq:yang_pdiff_cz_minus_v_lp} where the objective is now the Frobenius norm $\|\Gamma\|_F\triangleq\sqrt{\sum_{i=1}^{\ndvec{C}}\sum_{j=1}^{\ndvec{S}}
\Gamma_{ij}^2}$,
\begin{align}
    \hspace*{-0.75em}\begin{array}{cl}
    \underset{\Gamma\in\bbr^{\ndvec{C}\times \ndvec{S}}}{\text{minimize}} &\ \ {\|\Gamma\|}_F\\
    \text{subject\ to} &\ \ [G_C;A_C]\Gamma=[G_S;0_{\nconst{C}\times\ndvec{S}}].
    \end{array}\label{eq:yang_frobenius}%
\end{align}
When $\calc$ is full-dimensional, the optimal solution of \eqref{eq:yang_frobenius} is 
available in closed-form, $\Gamma^\ast=[G_C;A_C]^\dagger[G_S;0_{\nconst{C}\times\ndvec{S}}]$, by Prop.~\ref{prop:Indep}.2 and~\cite[Ex.
16.2]{boyd_introduction_2018}. Observe that $\Gamma^\ast=\Gamma G_S$ for $\Gamma$ prescribed by
Prop.~\ref{prop:pdiff_general_l2}.
Then, we can recover $D$ prescribed by \eqref{eq:D_ii_zonotope} in Corr.~\ref{corr:specific} using
\eqref{eq:diag_matrix_general} in Thm.~\ref{thm:pdiff_general}, where
$\rho_{\cals_0}(\nu)=\|\nu\|_1$ instead of
\eqref{eq:support_zonotope}.

The condition $D_{ii}\geq 0$ for every $i\in\Nint{1}{\ndvec{C}}$ is required in the two-stage approach~\cite{yang_efficient_2022} (see
\eqref{eq:yang_pdiff_cz_minus_v_lp}) and in Step~\ref{step:check} of Algo.~\ref{algo:inner_approx}. 
Thus, 
Algo.~\ref{algo:inner_approx} and~\cite{yang_efficient_2022} differ only
in the choice of $\Gamma$ (or specifically, the choice of objective in \eqref{eq:yang_pdiff_cz_minus_v_lp}) for a zonotopic subtrahend. 
In other words, Algo.~\ref{algo:inner_approx} may be viewed as a generalization of~\cite{yang_efficient_2022} for symmetric, convex, and compact subtrahends, and does not require an optimization solver for certain subtrahends (Corr.~\ref{corr:specific}). 
We compare these approaches in various examples in Sec.~\ref{sec:num}.

\section{Outer-approximation of Pontryagin difference and sufficient conditions for exactness}
\label{sec:outer_approx}

In this section, we first address Prob.~\ref{prob_st:outer_approx}, and then use its solution to address the
outer-approximation part of Prob.~\ref{prob_st:pdiff}. We also provide sufficient conditions under which all proposed approximations are exact. We conclude this section with a discussion of implementation considerations for the proposed
algorithms.

\subsection{Outer-approximating convex polyhedron for a given constrained zonotope}

\begin{algorithm}[t]
\caption{Convex polyhedral outer-approximation of a constrained zonotope
$\calc$}\label{algo:outer_approx_hrep}
\begin{algorithmic}[1]
    \Require Full-dimensional $\calc=(G_C,c_C,A_C,b_C)$ 

    \Ensure Convex polyhedron $\calp=\{x\mid Hx\leq k\}\supseteq \calc$
    
    \State $V\gets[v_1;v_2; \ldots;
    v_{\ndvec{C}}]\in\bbr^{\ndvec{C}\times(n+\nconst{C})}$ with \label{step:nu_i_defn}
    \begin{align}
        v_i = \frac{e_i^\top{[G_C;A_C]}^\dagger}{{\left\|{e_i^\top{[G_C;A_C]}^\dagger[G_C;A_C]}\right\|}_1},\
        \forall i\in\Nint{1}{\ndvec{C}}\label{eq:nu_i_defn}
    \end{align}
    
    \State $H\gets [V;-V][I_n;0_{\nconst{C}\times n}]$\label{step:H_defn} %

    \State $k\gets 1_{2\ndvec{C}\times
    1} - [V;-V][-c_C;b_C]$\label{step:k_defn} %

  \end{algorithmic}
\end{algorithm}

We now address Prob.~\ref{prob_st:outer_approx} using Algo.~\ref{algo:outer_approx_hrep}. 
\begin{proposition}
    Given a full-dimensional constrained zonotope $\calc$,
    Algo.~\ref{algo:outer_approx_hrep} computes a   convex, polyhedral, outer-approximation of
    $\calc$ with at most $2\ndvec{C}$ linear constraints. 
    \label{prop:reverse}
\end{proposition}
See Sec.~\ref{app:proofs_aux_prop_reverse} for the proof. We used the observation that a full-dimensional $\calc$ may be expressed as the $1$-sublevel set of the optimal value function of a feasible linear program. 
Using strong duality, we obtain an outer-approximating convex polyhedron $\calp$, characterized by $2\ndvec{C}$ feasible solutions
\eqref{eq:nu_i_defn} to the corresponding dual problem.  

While~\cite[Prop.  3]{scott_constrained_2016} provides an exact H-Rep polytope  representation of $\calc$ using \emph{lifted zonotopes}, it may require a combinatorial number of hyperplanes, and may not be practical for large $n$ and/or $\ndvec{C}$. 
Instead,  Algo.~\ref{algo:outer_approx_hrep} outer-approximates $\calc$ with a convex
polyhedron $\calp$ that has at most $2\ndvec{C}$ halfspaces, and can be efficiently computed for $\calc$ with large $n$ and
$\ndvec{C}$ without relying on convex optimization solvers. Redundant inequalities in $\calp$ may be removed via linear programming~\cite{MPT3}, if desired.

\begin{remark}{\textsc{(H-Rep  outer-approximation)}}
    We can modify Algo.~\ref{algo:outer_approx_hrep} in the following ways:\newline
    1) Always return a H-Rep polytope $\calp\cap\{x \mid l\leq x \leq u\}$ with $l,u$ computed in \eqref{eq:invertible_zonotope}, where $\nineq{P}=2(\ndvec{C}+n)$.\newline
    2) Reduce approximation error by intersecting $\calp$ computed by Algo.~\ref{algo:outer_approx_hrep} with supporting hyperplanes \eqref{eq:support} evaluated along template directions (ray shooting)~\cite{gleason2021lagrangian}.\newline
    Both modifications incur additional computational cost.
\end{remark}

\subsection{Outer-approximation of the Pontryagin difference}

We now address the outer-approximation part of  Prob.~\ref{prob_st:pdiff}.

\begin{algorithm}[t]
\caption{Outer-approximation of 
$\calc\ominus\cals$}\label{algo:outer_approx}
\begin{algorithmic}[1]
    \Require Minuend $\calc=(G_C,c_C,A_C,b_C)$ that is full-dimensional and a convex and compact
    subtrahend $\cals$ that is symmetric about any $c_S\in\bbr^n$

    \Ensure Constrained zonotope $\calm^+\supseteq \calc\ominus \cals$
    
    \State Compute a convex polyhedron $\calp\supseteq \calc$ using Algo.~\ref{algo:outer_approx_hrep}\label{step:poly}.

    \State Compute a convex polyhedron
    $\calm^+_\text{p}=\calp\ominus\cals$ via \eqref{eq:pdiff_polytope}\label{step:M_poly}.

    \State Compute $\calm^+\gets (\calc - c_S)\cap \calm^+_\text{p}$ using \eqref{eq:intersection_CZ_H}
    \label{step:intersect}.

  \end{algorithmic}
\end{algorithm}

\begin{proposition}
    For a full-dimensional, constrained zonotopic minuend $\calc$ and a convex and compact subtrahend
    $\cals\subset \bbr^n$ that is symmetric about any $c_S\in\bbr^n$, Algo.~\ref{algo:outer_approx} returns a constrained zonotope $\calm^+\supseteq\calc\ominus\cals$. \label{prop:outer}
\end{proposition}
\begin{proof}
    By Prop.~\ref{prop:reverse} and \eqref{eq:pdiff_polytope}, $\calm^+_\text{p}$ constructed in
    Step~\ref{step:M_poly} of Algo.~\ref{algo:outer_approx} outer-approximates $\calc\ominus\cals$. Since $\cals$ is symmetric about $c_S$, $(\calc -
    c_S)$ is also an outer-approximation of $\calc\ominus\cals$ by \eqref{eq:pdiff}. 
    We obtain the outer-approximating
    constrained zonotope $\calm^+$ by intersecting
    these outer-approximations in Step~\ref{step:intersect} using \eqref{eq:intersection_CZ_H}.
\end{proof}

All steps of Algo.~\ref{algo:outer_approx} are available in closed-form, when the
support function of $\cals$ is known in closed-form. Similarly to Algo.~\ref{algo:inner_approx}, Algo.~\ref{algo:outer_approx} is optimization-free, i.e., it does not require convex optimization solvers, 
when $\cals$ is an ellipsoid,  a zonotope, or a convex unions of symmetric 
intervals (see \eqref{eq:support_sets}). 

\subsection{Sufficient conditions for exactness}
\label{sub:exactness}

We show that the proposed approximations become exact when using an \textsc{Invertible} representation.

\begin{proposition}{\textsc{(Sufficient conditions for exactness of  polyhedral
    cover)}}\label{prop:exact_reverse}
    For an \textsc{Invertible} representation $\calc$, 
    Algo.~\ref{algo:outer_approx_hrep} computes a H-Rep polytope $\calp=\calc$ that is identical to the H-Rep polytope in Prop~\ref{prop:Indep}.5.
\end{proposition}
\begin{proposition}{\textsc{(Sufficient conditions for exactness of Pontryagin difference)}}\label{prop:exact_general}
    For an \textsc{Invertible} representation $\calc$, Algo.~\ref{algo:inner_approx} and~\ref{algo:outer_approx} provide approximations $\calm^-$ and $\calm^+$ that 
    satisfy
    $\calm^-=\calm=\calc\ominus\cals=\calm^+$.
\end{proposition}
See
Sec.~\ref{app:proofs_aux_prop_exact_reverse} and~\ref{app:proofs_aux_prop_exact_general} for the proofs.
The exactness results (Prop.~\ref{prop:exact_reverse} and~\ref{prop:exact_general}) may be attributed to Prop.~\ref{prop:Indep}.5.
Also, for $\calc$ in an \textsc{Invertible} representation,  Step~\ref{step:M_poly} of Algo.~\ref{algo:outer_approx} computes a H-Rep polytope $\calm^+_\text{p}=\calc\ominus\cals$ by Prop.~\ref{prop:exact_reverse}. 
Consequently, instead of Step~\ref{step:intersect}, we can use Algo.~\ref{algo:invertible} to compute $\calm^+$ directly from $\calm_p^+$.

Algo.~\ref{algo:inner_approx},~\ref{algo:outer_approx} address
Prob.~\ref{prob_st:pdiff} by Thm.~\ref{thm:pdiff_general} and Prop.~\ref{prop:pdiff_general_l2},~\ref{prop:outer},~\ref{prop:exact_general}. Algo.~\ref{algo:outer_approx_hrep} addresses
Prob.~\ref{prob_st:outer_approx}  by 
Prop.~\ref{prop:reverse},~\ref{prop:exact_reverse}.

\subsection{Implementation considerations}
\label{sub:implementation}

The use of pseudoinverse $[G_C;A_C]^\dagger$ in Algo.~\ref{algo:inner_approx},~\ref{algo:outer_approx_hrep},
and~\ref{algo:outer_approx} was motivated by providing closed-form expressions for $D_{ii}$
\eqref{eq:diag_matrix_general} and $v_i$ \eqref{eq:nu_i_defn}. However, the computation of $[G_C;A_C]^\dagger$ can be computationally expensive for large $\ndvec{C}$ and $\nconst{C}$.
In practice, it suffices to compute a minimum norm solution of systems of linear equations --- a solution $\Gamma$ in \eqref{eq:Gamma_linear_equations} for Step~\ref{step:Gamma} of Algo.~\ref{algo:inner_approx}, a solution $\dvec$ to $A_C\dvec=b_C$ in Step~\ref{step:check} of Algo.~\ref{algo:inner_approx}, and
a solution $V^\top$ to $[G_C;A_C]^\top V^\top=I_{\ndvec{C}}$ for
Step~\ref{step:nu_i_defn} of Algo.~\ref{algo:outer_approx_hrep}
(where
$V[G_C;A_C]$ is later normalized row-wise in
$\ell_1$-norm).

We can use QR
factorization or complete orthogonal decomposition to compute a minimum norm solution without explicitly computing the psuedoinverse~\cite[Ch.
12]{boyd_introduction_2018}. 
Existing algorithms can also exploit sparsity~\cite[Ch. 12.3]{boyd_introduction_2018}. In Sec.~\ref{sec:num} and~\ref{sec:num_app}, our 
\texttt{MATLAB} implementation of Algo.~\ref{algo:inner_approx},~\ref{algo:outer_approx_hrep},
and~\ref{algo:outer_approx} utilizes \texttt{lsqminnorm} to compute minimum norm solutions and uses sparse matrices for computational efficiency.

\section{Inner-approximation of robust controllable sets}
\label{sec:inner_approx_RC}

We now address Prob.~\ref{prob_st:RC_set} by inner-approximating the $T$-step RC set using
Algo.~\ref{algo:inner_approx} and the set
recursion \eqref{eq:set_recursion}. We consider both cases described in Prob.~\ref{prob_st:RC_set}, characterize the representation
complexity of the computed inner-approximations, and show that their representation complexities grow linearly with $T$.

Throughout this section, we will assume that the input set $\calu_t$ and the goal set $\calg$ are polytopes, hence, representable as constrained zonotopes, and the additive disturbance set $\calw_t$ is a convex and compact set that is
symmetric about any $c_W\in\bbr^p$.
Also, we assume that the sets $\calk_t$ computed in \eqref{eq:set_recursion} are full-dimensional for every $t$.

\subsection{Convex polyhedral $\calx_t$ and invertible $A_t$}
\label{sub:set_recursion_inv_A}

\begin{table}
    \centering
    \begin{tabular}{|c|c|c|c|}
        \hline
      Set & Eq. no.  & $\nconst{}$ & $\dofo{}$  \\\hline\hline
      $\calk_{t,\text{inner}}^\text{interim,1}$ & \eqref{eq:set_recursion_inv_A_pdiff}    & $\nconst{\calk_{t+1}}$ & $\dofo{\calk_{t+1}}$\\\hline
      $\calk_t^\text{interim,2}$  & \eqref{eq:set_recursion_inv_A_msum}   & \multirow{2}*{$\nconst{\calk_{t+1}}+\nconst{B\calu}$} & \multirow{3}*{$\dofo{\calk_{t+1}}+\dofo{B\calu}$} \\\cline{1-2}
      $\calk_t^\text{interim,3}$  & \eqref{eq:set_recursion_inv_A_invAmultiply} &  & \\\cline{1-3}
      $\calk_t$  & \eqref{eq:set_recursion_inv_A_intersect}   & $\nconst{\calk_{t+1}}+\nconst{B\calu}+\nineq{X}$ & \\\hline
    \end{tabular}
    \caption{Representation complexity (see Defn.~\ref{defn:repr}) for various sets involved in computing an inner-approximation to the $T$-step RC set using \eqref{eq:set_recursion_inv_A}, where
        $\scrc(\calk_{t+1})=(\nconst{\calk_{t+1}},\dofo{\calk_{t+1}})$ and $\scrc(B\calu)=(\nconst{B\calu},\dofo{B\calu})$, and
        $\calx$ is characterized by $\nineq{X}$ hyperplanes. Observe that the
    representation complexity grows by $(\nconst{BU}+\nineq{X}, \dofo{B\calu})$ with each step of the recursion.
    }
    \label{tab:rc_inv_A}
\end{table}

Given a finite horizon $T\in\bbn$, we consider the case where  $A_t$ in \eqref{eq:ltv_dyn} are invertible, and 
state constraints $\calx_t$ are polyhedra for all $t\in\Nint{0}{T-1}$. Since $\calx_t$ can be unbounded, they may not be representable as constrained zonotopes. However, we can still inner-approximate the RC sets as constrained zonotopes using \eqref{eq:intersection_CZ_H}.

For all $t\in\Nint{0}{T-1}$, we break down the set recursion \eqref{eq:set_recursion} to compute the $T$-step RC set into four steps:
\begin{subequations}
\begin{align}
    \calk_{t,\text{inner}}^\text{interim,1}&\subseteq\calk_t^\text{interim,1} = \calk_{t+1} \ominus F_t\calw_t,\label{eq:set_recursion_inv_A_pdiff}\\
    \calk_t^\text{interim,2} &= \calk_{t,\text{inner}}^\text{interim,1} \oplus (-B_t\calu_t),\label{eq:set_recursion_inv_A_msum}\\
    \calk_t^\text{interim,3} &= A_t^{-1}\calk_t^\text{interim,2},\label{eq:set_recursion_inv_A_invAmultiply}\\
    \calk_t &= \calk_t^\text{interim,3} \cap \calx_t\label{eq:set_recursion_inv_A_intersect}.
\end{align}\label{eq:set_recursion_inv_A}%
\end{subequations}
The recursion \eqref{eq:set_recursion_inv_A} is initialized with a constrained zonotope $\calk_T=\calg$. We use Prop.~\ref{prop:pdiff_general_l2} to compute a constrained zonotope $\calk_{t,\text{inner}}^\text{interim,1}$ in \eqref{eq:set_recursion_inv_A_pdiff}. Then, we compute
\eqref{eq:set_recursion_inv_A_msum}--\eqref{eq:set_recursion_inv_A_intersect} exactly using
\eqref{eq:set_operations_CZ}. Thus, for all $t\in\Nint{0}{T-1}$, the sets
$\calk_t,\calk_{t,\text{inner}}^\text{interim,1},\calk_t^\text{interim,2}$, and $\calk_t^\text{interim,3}$ are constrained zonotopes, and $\calk^-=\calk_0$ is an inner-approximation of the $T$-step RC set.

Table~\ref{tab:rc_inv_A} describes the representation complexity of various constrained zonotopes involved at each
step of \eqref{eq:set_recursion_inv_A}. For ease of discussion, we assume that $\calx_t,\calu_t$, and $\calw_t$ are time-invariant, i.e., $\calx_t=\calx$, $\calu_t=\calu$, and $\calw_t=\calw$, for all $t\in\Nint{0}{T}$. 
Additionally, we assume that $\calx$ is characterized by $\nineq{X}$ hyperplanes.  
Given representation complexities $\scrc(\calk_{t+1})=(\nconst{\calk_{t+1}},\dofo{\calk_{t+1}})$ and $\scrc(B\calu)=(\nconst{B\calu},\dofo{B\calu})$, the rows of Table~\ref{tab:rc_inv_A} follow from Prop.~\ref{prop:pdiff_general_l2}, \eqref{eq:msum_CZ}, 
\eqref{eq:affinemap_CZ}, and \eqref{eq:intersection_CZ_H}, respectively. 
With $\scrc(\calg)=(\nconst{\calg},\dofo{\calg})$, the representation complexity of the inner-approximation of the $T$-step RC set is 
\begin{align}
    \scrc(\calk^-)=(\nconst{\calg}+T(\nconst{B\calu}+\nineq{X}), \dofo{\calg} + T\dofo{B\calu}). \label{eq:set_complexity_inv_A}
\end{align}
Observe that the representation complexity of $\calk^-$ does not depend on the disturbance set $\calw$ due to
Prop.~\ref{prop:pdiff_general_l2}, and grows linearly with $T$.

\subsection{Polytopic $\calx_t$}
\label{sub:set_recursion_no_inv_A}

We now consider the case where the state constraints $\calx_t$ are polytopes for all $t\in\bbn$, and admit a constrained zonotope representation with $\scrc(\calx_t)=(L_{\calx_t}, 1)$ when using Algo.~\ref{algo:invertible}. Unlike Sec.~\ref{sub:set_recursion_inv_A}, we no longer assume that $A_t$
is invertible. 

Similarly to \eqref{eq:set_recursion_inv_A}, we separate the set recursion \eqref{eq:set_recursion} into \emph{three} steps
performed for all $t\in\Nint{0}{T-1}$:
\begin{subequations}
\begin{align}
    \calk_{t,\text{inner}}^\text{interim,1}&\subseteq\calk_t^\text{interim,1} = \calk_{t+1} \ominus F_t\calw_t,\label{eq:set_recursion_no_inv_A_pdiff}\\
    \calk_t^\text{interim,2} &= \calk_{t,\text{inner}}^\text{interim,1} \oplus (-B_t\calu_t),\label{eq:set_recursion_no_inv_A_msum}\\
    \calk_t &= \calx_t\cap_{A_t}\calk_t^\text{interim,2} \label{eq:set_recursion_no_inv_A_intersect},
\end{align}\label{eq:set_recursion_no_inv_A}%
\end{subequations}
where \eqref{eq:set_recursion_no_inv_A_intersect} combines \eqref{eq:set_recursion_inv_A_invAmultiply} and \eqref{eq:set_recursion_inv_A_intersect} into a single step using \eqref{eq:intersection_CZ}.
Assuming time-invariance, $\scrc(\calx)=(L_X,1)$ and the representation complexity of $\calk_t$ grows by $(\nconst{BU}+L_X+n, \dofo{B\calu})$ at each step of the set recursion (see \eqref{eq:intersection_CZ} and Table~\ref{tab:rc_inv_A}). 
Consequently, with $\scrc(\calg)=(\nconst{\calg},\dofo{\calg})$, the representation complexity of the inner-approximation of the $T$-step RC set is 
\begingroup
    \makeatletter\def\f@size{9.5}\check@mathfonts
\begin{align}
    \scrc(\calk^-)=(\nconst{\calg}&+T(\nconst{B\calu}+ L_X + n), \dofo{\calg} + T\dofo{B\calu}). \label{eq:set_complexity_no_inv_A}
\end{align}
\endgroup
Similarly to \eqref{eq:set_complexity_inv_A}, $\scrc(\calk^-)$ in \eqref{eq:set_complexity_no_inv_A} also grows linearly with $T$ and does not depend on the disturbance set
$\calw$.

\begin{remark}\label{rem:outer}
    We can also use the set recursions discussed in Sec.~\ref{sub:set_recursion_inv_A}
    and~\ref{sub:set_recursion_no_inv_A} in conjunction with Algo.~\ref{algo:outer_approx} to obtain an
    outer-approximation to the $T$-step RC set.
\end{remark}
\begin{remark}\label{rem:reduce}
In this work, we did not use the exact or approximate reduction
techniques~\cite{raghuraman2022set,scott_constrained_2016,kopetzki2017methods}
that may lower the representation complexity at the
expense of additional computation or accuracy or both. On the other hand, it is straightforward to apply these results to the sets computed in this work for a further reduction in the set representation complexity. 
\end{remark}

\section{Case studies}
\label{sec:num}

We now demonstrate the computational efficiency, scalability, and utility of our approach in several case studies.
First, we consider two case studies involving low-dimensional systems to illustrate the advantage of the proposed approach in
computing time when compared to existing inner-approximating approaches based on constrained
zonotopes~\cite{yang_efficient_2022}, and exact approaches based on H-Rep and V-Rep polytopes~\cite{MPT3}. 
Then, we discuss the scalability of the
approach by computing the RC set for a chain of mass-spring-damper  system, and empirically demonstrate that the proposed approach
is numerically stable and can compute the RC sets for high-dimensional systems and long horizons.

We perform the presented computations in a standard computer with Intel CPU i9-12900KF processor (3.2 GHz, 16 cores) and 64 GB RAM, running MATLAB 2022b on Windows. We use YALMIP~\cite{YALMIP}, \text{MOSEK}~\cite{MOSEK}, and \text{GUROBI}~\cite{GUROBI} to set up and solve the optimization problems. For two-dimensional plots of constrained zonotopes $\calc$, we compute the appropriate H-Rep/V-Rep polytope approximations via support function and vector computations \eqref{eq:support} in $100$ equi-spaced directions in $\bbr^2$. We estimate the volume of the sets via grid-based sampling.

\subsection{Double integrator example with polytopic $\calx$}
\label{sub:ex1}

We consider the computation of RC set for a double integrator system with polytopic state constraints $\calx$ and an ellipsoidal disturbance set $\calw$. The linear time-invariant system matrices are
\begin{align}
    A &= \left[\begin{array}{cc}
         1 & \Delta T  \\
         0 & 1
    \end{array}\right],\quad B = \left[\begin{array}{c}
        {(\Delta T)}^2/2 \\
         \Delta T
    \end{array}\right],\text{ and }F=I_2, \nonumber
\end{align}
with the sampling period $\Delta T = 0.1$, the input set $\calu =\llbracket -2, 2\rrbracket$ that is an interval, the disturbance set $\calw=(0.1 I_2, [0;0])$ or $\calw=(\text{diag}([0.2, 0.04]), [0.1;0.1])\}$ that is 
a circle or an ellipsoid respectively, and
the state constraints $\calx=\calg=\llbracket -2, 2\rrbracket\times \llbracket -3, 3\rrbracket$ that are time-invariant, axis-aligned rectangles.

We generate inner-approximations of the $T$-step RC set using the recursion in Sec.~\ref{sub:set_recursion_no_inv_A}
for $T=20$ with an exact ellipsoidal representation of $\calw$ and a zonotopic outer-approximation $\calw^+=(G_W,
c_W)\supset \calw$ where the Pontryagin difference is inner-approximated using Algo.~\ref{algo:inner_approx}.
We also compare the computed sets with the exact sets computed using MPT3 where the vertex-facet enumeration was
accomplished using Fourier-Motzkin elimination~\cite{MPT3}, and the
inner-approximations of the RC sets using the two-stage approach with the zonotopic
$\calw^+$~\cite{yang_efficient_2022}. We also compute an outer-approximation of the RC sets using Algo.~\ref{algo:outer_approx}, as discussed in Rem.~\ref{rem:outer}.

Table~\ref{tab:comparison_di} shows that the proposed inner-approximation approach with ellipsoidal $\calw$ is about
two orders of magnitude faster than the exact approach~\cite{MPT3} and the $\calw^+$-based approximation via two-stage approach~\cite{yang_efficient_2022}, while providing reasonably accurate inner-approximations (about $97\%$ and $77\%$ of the area of the exact RC set in the first and second case  respectively). 
Our approach with zonotopic $\calw^+$ is slightly faster than with ellipsoidal $\calw$, but
it is more conservative due to the zonotopic outer-approximation of the disturbance set. 
The shorter computation times of the proposed approach are a direct result of Thm.~\ref{thm:pdiff_general} and Prop.~\ref{prop:pdiff_general_l2}, since the implementation of the set recursion in Sec.~\ref{sub:set_recursion_no_inv_A} can be accomplished in closed-form, i.e., optimization-free. 
As expected, $(\nconst{\calk_0},\dofo{\calk_0})$ of the inner-approximations do not depend on choice of $\calw$ (see \eqref{eq:set_complexity_no_inv_A}).
The proposed outer-approximation is slower than the inner-approximation, primarily due to the use of linear programming to produce a minimal representation of the convex polyhedron in Step~\ref{step:poly} of Algo.~\ref{algo:outer_approx} to manage the representation complexity. 
Additionally, $\nconst{\calk_0}$ for outer-approximation is higher than the inner-approximation due to Step~\ref{step:poly} of Algo.~\ref{algo:outer_approx}.
For most safe constrained control problems, an inner-approximation of the $T$-step RC set is sufficient.

Fig.~\ref{fig:double_integrator} shows that the RC sets and their corresponding inner-approximations, associated with
a ball-shaped $\calw$ (left) and an ellipsoidal $\calw$ (right). As expected, the inner-approximations of the RC set
constructed using zonotopic $\calw^+$ are more conservative than their ellipsoid-based counterparts in both cases.
On the other hand, the proposed inner-approximations with zonotopic $\calw^+$ are identical (left) or similar (right) to
the inner-approximations produced by the existing two-stage approach~\cite{yang_efficient_2022}, while requiring
significantly shorter computation time (see Table~\ref{tab:comparison_di}).

\begin{table}
    \centering
    \begin{tabular}{|c||r|r|r|r|r||r|r|r|r|r|}\hline
    \multirow{3}{*}{Method} & Area & \multicolumn{2}{c|}{Compute time} & \multicolumn{2}{c||}{Complexity $\scrc$} & Area & \multicolumn{2}{c|}{Compute time} & \multicolumn{2}{c|}{Complexity $\scrc$} \\\cline{2-11}
    & Ratio & Time (s) & Ratio & $\nconst{\calk_0}$ & $\dofo{\calk_0}$ & Ratio & Time (s) & Ratio & $\nconst{\calk_0}$ & $\dofo{\calk_0}$ \\\cline{2-11}
    & \multicolumn{5}{c||}{$\calw$ is a ball (left in Fig.~\ref{fig:double_integrator})} 
    & \multicolumn{5}{|c|}{$\calw$ is an ellipsoid (right in Fig.~\ref{fig:double_integrator})} \\\hline
    Exact~\cite{MPT3} & 1 & 2.044 & 177.39 & \multicolumn{2}{c||}{N/A} 
    & 1 & 2.958 & 176.65 & \multicolumn{2}{c|}{N/A}\\\cline{5-6}\cline{10-11}
    Ours $\calm^-$ & \textbf{0.97} & 0.012 & 1 & 120 & 11
    & \textbf{0.77} & 0.017 & 1 & 120 & 11\\
    Ours $\calm^+$ & 1.67 & 4.064 & 352.66 & 322 & 11
    & 3.46 & 4.199 & 250.78 & 326 & 11\\\cline{2-11}
    & \multicolumn{10}{c|}{With a zonotope outer-approximation $\calw^+\supset\calw$ as the disturbance set}\\\cline{2-11}
    Ours $\calm^-$ & 0.70 & \textbf{0.010} & \textbf{0.86} & 120 & 11
    & 0.37 & \textbf{0.014} & \textbf{0.84} & 120 & 11\\
    2-stage~\cite{yang_efficient_2022} & 0.67 & 1.594 & 138.29 & 120 & 11
    & 0.22 & 1.520 & 90.78 & 120 & 11\\\hline
    \end{tabular}
    \caption{Comparison of various approaches for Sec.~\ref{sub:ex1}. 
    Our (inner-approximation) approach is about two orders of magnitude faster than existing approaches~\cite{MPT3,yang_efficient_2022}, and generates sufficiently accurate approximations. 
    }
    \label{tab:comparison_di}
\end{table}

\begin{figure}
    \centering
    \includegraphics[width=0.35\linewidth,trim={450 100 0 0}, clip]{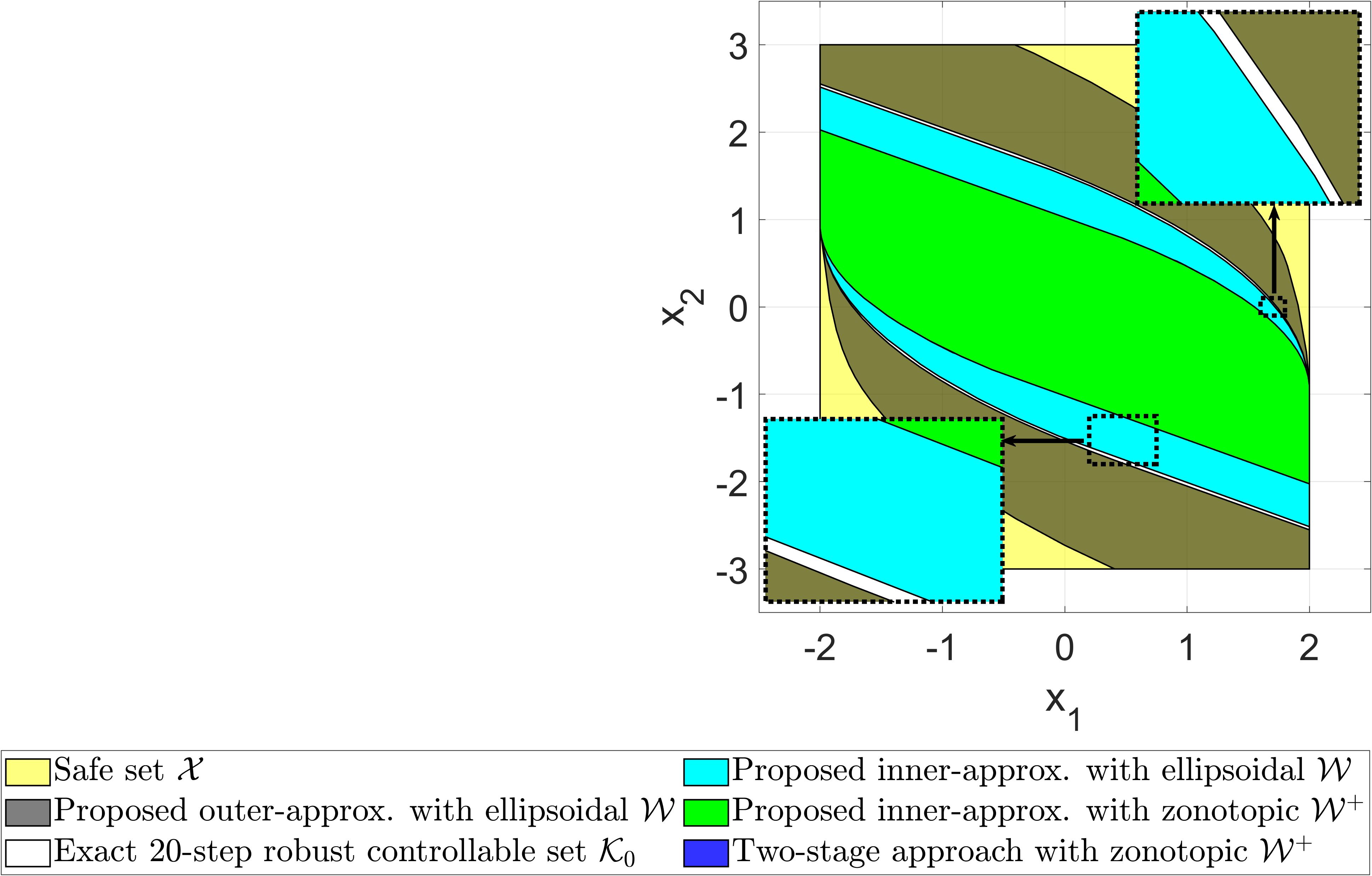}\ 
    \includegraphics[width=0.35\linewidth,trim={450 100 0 0}, clip]{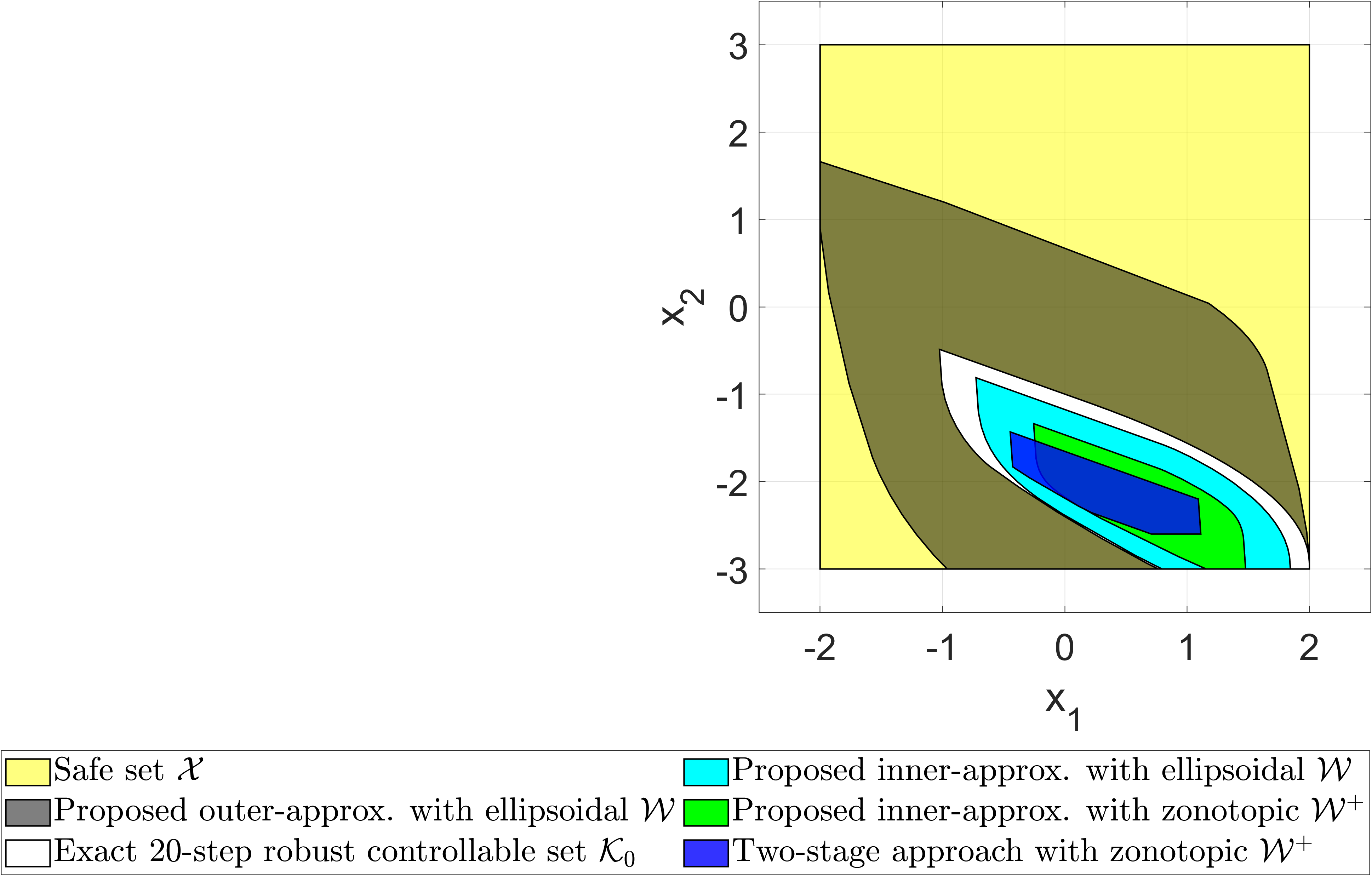}\\
    \includegraphics[width=0.8\linewidth,trim={0 0 0 500}, clip]{figures/ex1/sets_shifted.png}
    \caption{Robust controllable sets  computed using the recursion in
Sec.~\ref{sub:set_recursion_no_inv_A} for Sec.~\ref{sub:ex1} 
    with a ball-shaped $\calw$
    (left) and an ellipsoidal $\calw$ (right).
    We compare the sets obtained with the proposed approximations to the sets from existing approaches (exact~\cite{MPT3} and the two-stage approach~\cite{yang_efficient_2022}). 
    The insets in the left figure show that the proposed approach with ellipsoidal $\calw$ (cyan) provides sufficiently accurate inner-approximations of the exact RC set (white).}
    \label{fig:double_integrator}
\end{figure}

\subsection{An example with convex polyhedral $\calx$}
\label{sub:ex2}

We now consider the computation of RC set over a long horizon
$T=100$, similar to~\cite[Ex. 2]{yang_efficient_2022}. Consider a stable, discrete-time, linear time-invariant system with matrices
\begin{align}
    A &= \left[\begin{array}{cc}
         0.99 & 0.02  \\
         -0.15 & 0.99
    \end{array}\right],\quad B = \left[\begin{array}{c}
        -0.01 \\
         0.08
    \end{array}\right],\text{ and }F=I_2, \nonumber
\end{align}
and an interval input set $\calu =\llbracket -1.5, 1.5\rrbracket$, zonotopic disturbance set $\calw=(0.01 I_2, [0;0])$, 
zonotopic goal set $\calg=(0.5 I_2, [1.5;0])$, and convex polyhedral, time-invariant state constraints $\calx=\{x \mid [-1, 0; 2,    1] x \leq  [2;5]\}$.

Since the state constraints are polyhedral and $A$ is invertible, we generate inner-approximations of the
$T$-step RC set using the recursion in Sec.~\ref{sub:set_recursion_inv_A}. Similarly to
Sec.~\ref{sub:ex1}, we compare the obtained sets with their exact counterparts computed using MPT3~\cite{MPT3} and the
inner-approximations obtained using the two-stage approach~\cite{yang_efficient_2022}. We also compute an outer-approximation of the RC sets using
Algo.~\ref{algo:outer_approx}.

Table~\ref{tab:comparison_yang} shows that the proposed inner-approximating approach
with zonotopic $\calw$ is over two to three orders of magnitude faster than the existing approaches~\cite{MPT3,yang_efficient_2022}, while providing
reasonably accurate inner-approximations that cover about $89\%$ of the area of the exact RC set, as in Sec.~\ref{sub:ex1}. The shorter computation time compared to existing approaches is attributed to the optimization-free implementation of
Sec.~\ref{sub:set_recursion_inv_A}, by  Prop.~\ref{prop:pdiff_general_l2} and
Corr.~\ref{corr:specific}.
As expected, $(\nconst{\calk_0},\dofo{\calk_0})$ for our $\calm^-$ and the two-stage approach are identical~\cite{yang_efficient_2022}.

\begin{table}
    \centering
    \begin{tabular}{|c||r|r|r|r|r|}\hline
    \multirow{2}{*}{Method} & Area & \multicolumn{2}{c|}{Compute time} & \multicolumn{2}{c|}{Complexity $\scrc$} \\\cline{2-6}
    & Ratio & Time (s) & Ratio & $\nconst{\calk_0}$ & $\dofo{\calk_0}$ \\\hline
    Exact~\cite{MPT3} & 1 & 110.661 & 1235.87 & \multicolumn{2}{c|}{N/A}\\\cline{5-6}
    Ours $\calm^-$ & 0.89 & \textbf{0.090} & \textbf{1} & 200 & 202.0\\
    Ours $\calm^+$ & 1.36 & 158.753 & 1772.96 & 1703 & 953.5\\
    2-stage~\cite{yang_efficient_2022} & \textbf{0.92} & 10.192 & 113.82 & 200 & 202.0\\\hline
    \end{tabular}
    \caption{Comparison of various approaches for Sec.~\ref{sub:ex2}. 
    Our (inner-approximation) approach is about two to three orders of magnitude faster than existing approaches~\cite{MPT3,yang_efficient_2022}, and generates accurate approximations  even for a long horizon $T$.}
    \label{tab:comparison_yang}
\end{table}
\begin{figure}[t]
    \centering\includegraphics[width=0.8\linewidth, trim={100 80 70 45}, clip]{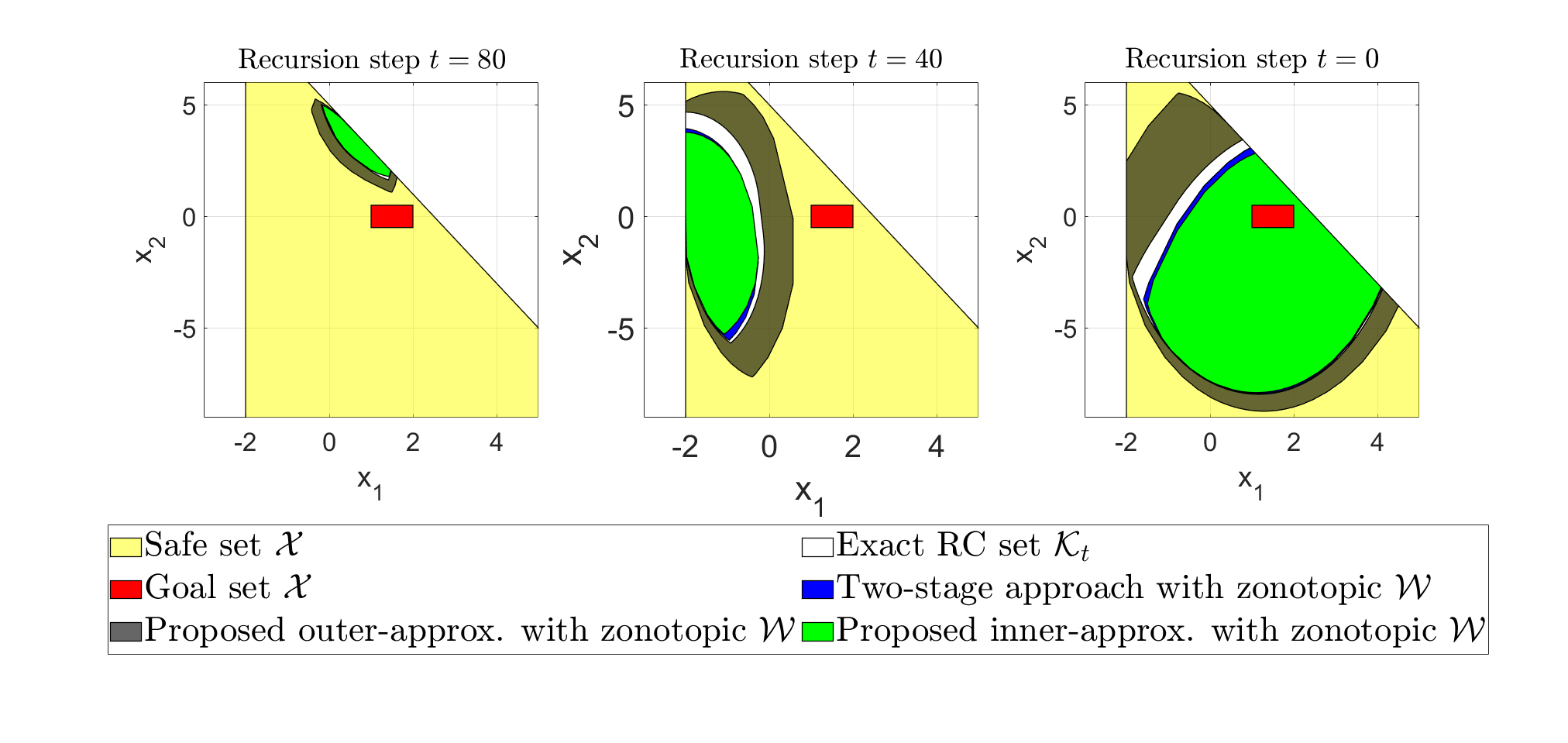}
    \caption{Snapshots of the $100$-step robust controllable sets computed using Sec.~\ref{sub:set_recursion_inv_A} for Sec.~\ref{sub:ex2} at recursion steps $t\in\{0, 40, 80\}$. 
 Our approach computes inner-approximations of the
RC set that are similar to those obtained using the exact approach~\cite{MPT3} and the two-stage  approach~\cite{yang_efficient_2022}, with significantly lower computational effort (see Table~\ref{tab:comparison_yang}).}
    \label{fig:comparison_yang}
\end{figure}
Fig.~\ref{fig:comparison_yang} shows the $100$-step RC sets computed by various methods at recursion step $t\in\{0, 40,
80\}$. Unlike in Sec.~\ref{sub:ex1}, the proposed inner-approximation of the RC set in this case is contained in the
inner-approximation using the two-stage approach~\cite{yang_efficient_2022}. However, the conservativeness
of the proposed approach compared to the two-stage approach~\cite{yang_efficient_2022} appears minimal, covering about $89\%$ vs $92\%$ of the area of the exact RC set.

\subsection{Scalability: Chain of damped spring-mass systems}
\label{sub:ex3}

We now demonstrate scalability of the representation complexity for the proposed inner-approximation. Specifically, we compute the RC set
of a chain of $\Nmass\in\bbn$ homogenous mass-spring-damper systems, see Fig.~\ref{fig:spring_mass}, for a range of
chain lengths, i.e., different system dimension, and set  recursion lengths, i.e., different horizon $T$. The chain system has the following continuous-time linear time-invariant dynamics,
\begingroup
    \makeatletter\def\f@size{8.5}\check@mathfonts
    \begin{subequations}
\begin{align}
    \ddot{x}_1 &= -\frac{2k}{m}x_1 + \frac{k}{m} x_2 - \frac{\mu}{m} \dot{x}_1 + w_1 + u_1,\\
    \ddot{x}_\Nmass &= -\frac{2k}{m}x_\Nmass + \frac{k}{m} x_{(\Nmass-1)} - \frac{\mu}{m} \dot{x}_\Nmass \nonumber\\
                    &\qquad\qquad\hspace*{0.8em} + w_\Nmass + u_\Nmass,\\
    \ddot{x}_j &= -\frac{2k}{m}x_j + \frac{k}{m} (x_{j-1} + x_{j+1}) - \frac{\mu}{m} \dot{x}_j + w_j + u_j,
\end{align}\label{eq:chain_spring_mass}%
\end{subequations}%
\endgroup
where $j\in\Nint{2}{(\Nmass-1)}$. 
Here, \eqref{eq:chain_spring_mass} describes a $(2\Nmass)$-dimensional system parameterized by the spring constant $k$, the mass $m$, and the friction coefficient $\mu$. 
Each spring is actuated by an acceleration input $u_i\in\calu\subset\bbr$
subject to an acceleration  disturbance  
$w\in\calw\subset\bbr$. 
\begin{figure}[t]
    \centering\includegraphics[width=0.8\linewidth]{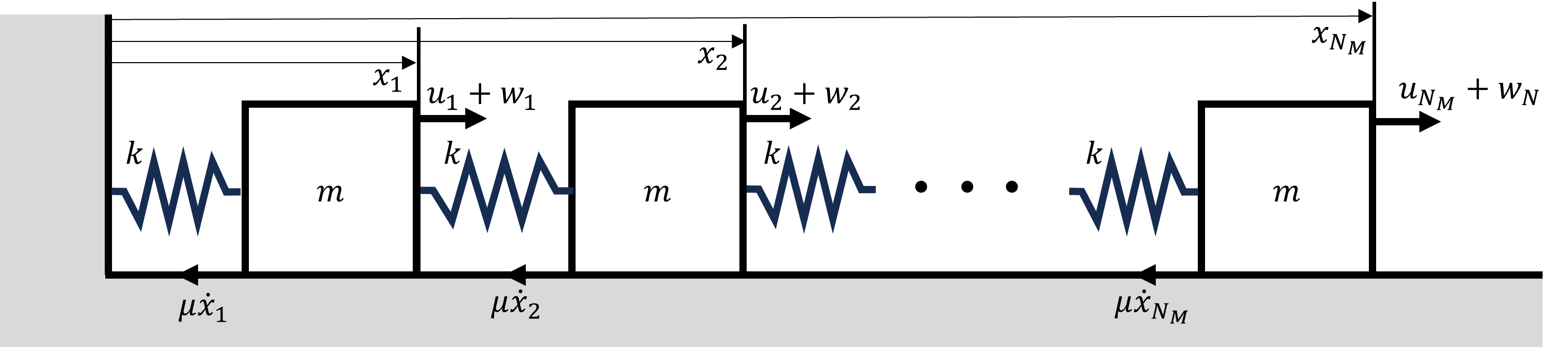}
    \caption{Chain of $\Nmass$ mass-spring-damper systems.}
    \label{fig:spring_mass}
\end{figure}

After discretizing \eqref{eq:chain_spring_mass} with sampling time $\Delta T=0.1$ using zero-order hold, we consider the
following computations of $T$-step RC sets:\\
1) $n\in\Nint{4}{100}$ with $T=20$ using Sec.~\ref{sub:set_recursion_no_inv_A},\\
2) $n\in\Nint{4}{100}$ with $T=40$ using Sec.~\ref{sub:set_recursion_no_inv_A},\\
3) $n\in\Nint{4}{50}$ with $T=20$ using two-stage approach~\cite{yang_efficient_2022},\\
4) $n\in\Nint{4}{14}$ with $T=20$ using exact approach~\cite{MPT3}.\\
We use parameters $k=0.1$, $m=0.1$, and $\mu=0.01$, input set $\calu={\llbracket-0.1,0.1\rrbracket}^{\Nmass}$,
disturbance set 
$\calw={\llbracket-0.0001,0.0001\rrbracket}^{\Nmass}$, and
state constraints
$\calx=\calg={\left({\llbracket-0.2,0.2\rrbracket\times\llbracket-0.5,0.5\rrbracket}\right)}^{\Nmass}$.

\begin{figure}
    \centering\includegraphics[width=0.8\linewidth,trim={0 0 0 0}, clip]{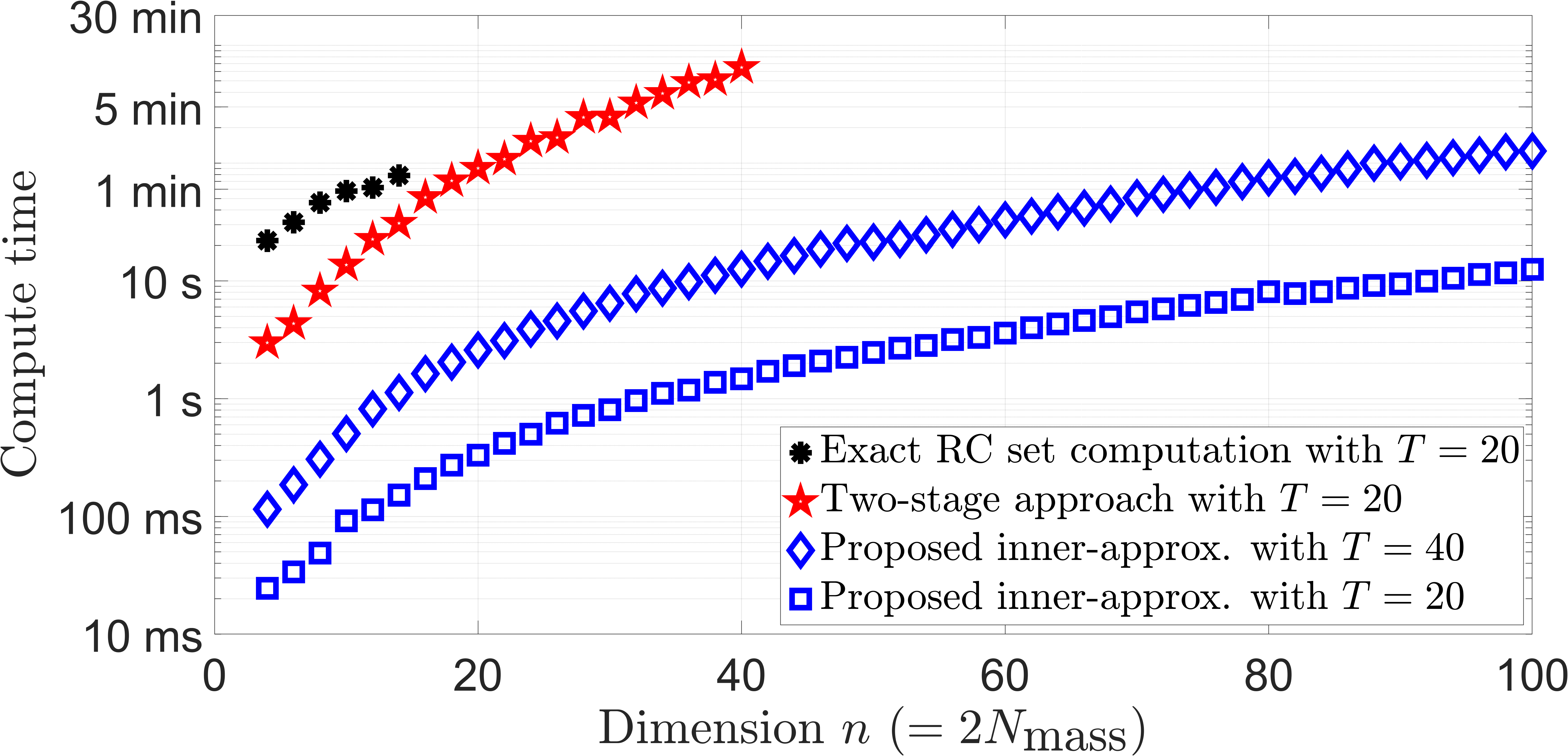}
    \caption{Time taken by various methods 
    to compute the RC sets 
    for varying system dimension $n$.
        The proposed method takes $12.52$ seconds to inner-approximate $20$-step RC sets for a $100$-dimensional system. In
        contrast, existing methods (the exact computation using MPT3~\cite{MPT3} and the two-stage approach in~\cite{yang_efficient_2022}) require longer computation time to tackle lower
        dimensional systems. We also report the time taken by the proposed method to inner-approximate $40$-step RC sets.}
    \label{fig:scalability}
\end{figure}

Fig.~\ref{fig:scalability} shows that the proposed method takes significantly shorter computation time to produce an
inner-approximation to the $20$-step RC sets, when compared to existing methods~\cite{yang_efficient_2022,MPT3}.
Specifically, we observe that our approach takes $12.52$ seconds to inner-approximate the $20$-step RC set for a
$100$-dimensional system ($\Nmass=50$). On the other hand, the two-stage approach~\cite{yang_efficient_2022} took $649.74$ seconds (about $10$ minutes) to compute an
inner-approximation for the $20$-step RC set for much smaller dimensional system $n=40$ ($\Nmass=20$). We encountered numerical issues for the
exact set computation using MPT3~\cite{MPT3} beyond $n=14$ ($\Nmass=7$). As expected, the proposed approach took longer
to compute the RC sets for an horizon $T=40$ compared to the sets for an horizon $T=20$, but still computed an inner-approximation to the $40$-step RC set for the $100$ dimensional system in $126.53$ seconds (about $2$ minutes). The scalability 
of the proposed approach compared to existing approaches may be attributed to the optimization-free implementation of
Sec.~\ref{sub:set_recursion_no_inv_A}, made possible by  Prop.~\ref{prop:pdiff_general_l2} and
Corr.~\ref{corr:specific}.

We oberved a moderate growth in the representation complexity of the proposed inner-approximations. 
For a $10$-dimensional system, an inner-approximating constrained zonotope
$\calk_0$ for the $20$-step RC set had a representation complexity of $\scrc(\calk_0)=(620, 11)$ with $n=10, \ndvec{\calk_0}=730, \nconst{\calk_0}=620$. As noted
in Rem.~\ref{rem:reduce}, various reduction techniques may be used to further lower the set
representation complexities, if so desired.

\section{Application: Abort-safe rendezvous} 
\label{sec:num_app}

Abort safety in spacecraft rendezvous~\cite{daniel,marsillach2021fail,vinod2021abort} requires that a spacecraft in nominal operation approaching a target must retain the ability to 
avoid collision with the target in the event of an anomaly or a failure. In~\cite{vinod2021abort}, we showed that the problem of abort-safe spacecraft rendezvous could be encoded using RC sets, and we computed these sets using H-Rep/V-Rep polytopes. However, such an approach is challenging in high-dimensions and suffers from numerical issues, which motivates the computation of the RC sets using constrained zonotopes. Additionally, to guarantee safety, we require an exact computation or inner-approximation of the RC sets.

In this work, we consider rendezvous to a future Lunar gateway flying in a near-rectilinear halo orbit (NRHO) around the Moon~\cite{gateway}. To minimize the use of fuel, we allow the spacecraft to utilize all 3 degrees of freedom in its approach. We compute six-dimensional RC sets that constrain the rendezvous trajectory in order to guarantee that, in the event of a failure, off-nominal operation of the spacecraft allows for a safe abort maneuver. Additionally, the trajectory must lie in a line-of-sight cone that arises from sensing and communication requirements, and ensures that the Sun stays behind the spacecraft to help in perception of the Lunar gateway. 

\emph{Nominal dynamics:} We obtain the unactuated  nonlinear dynamics of the spacecraft in the vicinity of the Lunar gateway by considering Earth and Moon's gravitational forces and dominant perturbations~\cite{muralidharan2020control}. We linearize the dynamics around the NRHO of the gateway, and discretize the dynamics in 
$T_\text{sample}$-long time intervals to obtain the relative dynamics~\cite{marsillach2021fail},
\begin{align}
x_{t+1} &= A_t x_t + B_t u_t, \label{eq:nom_ltv}
\end{align}
with state (position and velocity) $x_t\in\bbr^6$,  input $u_t\in\calu\subset\bbr^3$ models impulsive changes in velocities, 
and perfect state measurements. 

\emph{Off-nominal dynamics:} We consider three modifications to the nominal dynamics in the event of failure --- 1) limited actuation to model the event where the main thrusters fail and the spacecraft is forced to use redundant thrusters like attitude thrusters, 2) process noise to model the resulting actuation uncertainty, and 3) measurement noise to model sensing uncertainty that may increase with the use of redundant thrusters.  The need for redundant thrusters is well-known in space applications to ensure safety in off-nominal scenarios~\cite{fehse2003automated}. %
We assume that the process and measurement noises are drawn from pre-determined bounded sets that may be characterized via offline statistical analysis~\cite{fehse2003automated}.

Specifically, the off-nominal dynamics after 
a failure event at time $t$ are,
\begin{subequations}
\begin{align}
        z_{k+1|t} &= A_t z_{k|t} + B_t (u_{k|t} + w_{k|t}),\\
    \hat{z}_{k|t} &= z_{k|t} + \eta_{k|t},\label{eq:offnom_estimates}
\end{align}\label{eq:offnom_ltv}%
\end{subequations}%
where $z_{k|t}$ is the state after failure at time $k\geq t$ initialized by $z_{t|t}=x_t$ (the nominal state at failure time $t$), and $w_{k|t}\in\Woffnom\subset\bbr^3$ and
$\eta_{k|t}\in\Eoffnom\subset\bbr^6$ are bounded disturbances to the input and post-failure state respectively. 
The disturbances model the actuation mismatch and sensing limitations that can become prominent after failure. 
The post-failure input $u_t\in\Uoffnom\subset\calu$ where $\Uoffnom$ models the limited actuation available after failure. 
We consider a feedback controller 
$\pi:\bbr^6\to\Uoffnom$ that provides a post-failure control $u$ in \eqref{eq:offnom_ltv} given the current state estimate $\hat{z}\in\bbr^6$. Let $\Pi$ be the set of all such controllers.

From \eqref{eq:offnom_ltv}, the state estimate $\hat{z}_t$ follows the  dynamics,
\begin{align}
    \hat{z}_{t+1} &= A_t \hat{z}_t + B_t u_t + \phi_t,\label{eq:offnom_ltv_est}
\end{align}
with disturbance $\phi_t\in\Phi_t=\Eoffnom\oplus(B_t\Woffnom)\oplus (-A_t\Eoffnom)$.
From \eqref{eq:offnom_estimates},  $z_t$ and $\hat{z}_t$ satisfy 
\begin{align}
    z_t\in\hat{z}_t+(-\Eoffnom)\text{ and }
    \hat{z}_t\in z_t+\Eoffnom\label{eq:z_t_hat_z_t}.
\end{align}

\emph{Rendezvous constraints:} We consider the problem of navigating the spacecraft to a target set $\calt\subset\bbr^3$ in front of the Lunar
gateway, while staying inside a line-of-sight cone 
$\call\subset\bbr^3$ originating from the Lunar gateway.
The designed nominal rendezvous trajectory must also stay outside a keep-out set
$\cald\subset\bbr^3$ around the Lunar gateway during the rendezvous maneuver. 
Also, for some pre-determined post-failure safety horizon $\Tsafe\in\bbn$, the nominal state $x_t$ at any time $t$ must satisfy the abort-safety requirement:
\begingroup
    \makeatletter\def\f@size{8.5}\check@mathfonts
\begin{align}
    \text{(Abort-safety):}&\quad
    \left\{\begin{array}{l}
    \forall k\in\Nint{t}{t+\Tsafe},\ \exists \pi_k\in\Pi,\\
    \forall w_{k|t}\in\Woffnom,\ \forall \eta_{k|t}\in\Eoffnom,\\
    z_{k|t}\text{ in \eqref{eq:offnom_ltv} with } u_{k|t}=\pi_k(\hat{z}_{k|t})\\
    \text{satisfies }z_{k|t}\not\in\cald\text{, given $z_{t|t}=x_t$}.
\end{array}\right.\label{eq:abort_safety}
\end{align}
\endgroup
Informally, \eqref{eq:abort_safety} requires the  
nominal trajectory to permit steering the spacecraft to continue staying outside $\cald$ using limited actuation
and imperfect state information under
perturbed dynamics \eqref{eq:offnom_ltv} over a safety horizon of length $\Tsafe$, in the event of a failure at time
$t$. 

\emph{Optimal control problem:} Given initial state $x_0$, the optimal control problem is formulated as,
\begin{align}
    \begin{array}{rl}
    \text{min} &\quad\sum_{t} \left(\text{dist}(x_t, \calt)^2 + \lambda {\|u_t\|}_2\right)\\
    \text{s. t.} &\quad\text{Nominal dynamics \eqref{eq:nom_ltv} from  $x_0$},\\
    \forall t, &\quad x_t\in\call,\text{ and }x_t\not\in\cald\\
    \forall t, &\quad \text{$x_t$ meets abort-safety requirement \eqref{eq:abort_safety}}.\\
\end{array}\label{eq:prob_sketch}
\end{align}
For $\lambda>0$, \eqref{eq:prob_sketch} balances the typical goals of rendezvous
--- approaching the target
set $\calt$ while limiting the energy spent. We measure the energy spent as  
${(\Delta v)}_t= {\|u_t\|}_2$.

\emph{Enforcement of (Abort-safety) using RC sets:} Similarly to~\cite{vinod2021abort}, we encode the abort-safety requirements using appropriately defined RC sets. 
Let $\cals^\complement=\bbr^3\setminus\cals$ be the complement of a set
$\cals\subseteq\bbr^3$.
\begin{proposition}{\textsc{(Sufficient condition for Abort-safety)}}\label{prop:suff_abort_safety}
    Consider a H-Rep polytope keep-out set $\cald=\cap_{i=1}^{\nineq{D}} \calh_i$ with $\nineq{D}$ halfspaces
    $\calh_i\subset\bbr^3$, and $\Eoffnom$ that is symmetric about the origin. For any time $s\in\bbn$, let
    $\calk(s, \Tsafe, \calh_i^\complement\ominus\Eoffnom)$ denote the $\Tsafe$-step RC set for dynamics \eqref{eq:offnom_ltv_est}
characterized by ${\{(A_t,B_t,\Phi_t)\}}_{t=s}^{s+\Tsafe}$, $\Uoffnom$, and  
    $\calx_t=\calg=\calh_i^\complement\ominus\Eoffnom$. Then, $x_s$ satisfies
    \eqref{eq:abort_safety} if $x_s\in\cup_{i=1}^{\nineq{D}}\left({\calk(s, \Tsafe,
\calh_i^\complement\ominus\Eoffnom)\ominus\Eoffnom}\right)$.
\end{proposition}

See Sec.~\ref{app:proofs_aux_prop_suff_abort_safety} for the proof of Prop.~\ref{prop:suff_abort_safety} using \eqref{eq:offnom_ltv_est} and \eqref{eq:z_t_hat_z_t}. 

We solve \eqref{eq:prob_sketch} using a receding horizon framework.
For a finite planning horizon $\Tplan\in\bbn$,  the following (non-convex)
optimization problem approximates \eqref{eq:prob_sketch},
\begingroup
    \makeatletter\def\f@size{9}\check@mathfonts
\begin{align}
    \begin{array}{cl}
        \underset{\substack{x_{(t+1)|t},\ldots,x_{(t+\Tplan)|t}\\
        u_{t|t},\ldots,u_{(\Tplan-1)|t}}}{\text{minimize}} &\ \sum_{t} \left(\text{dist}(x_t, \calt)^2 + \lambda {\|u_t\|}_2\right)\\
        \text{subject\ to} &\ \text{Dyn. \eqref{eq:nom_ltv} defines $x_{k|t}$ given $x_t$},\\
    \forall k\in\Nint{t+1}{t+\Tplan}, &\ x_{k|t}\in\call,\ x_{k|t}\in\bigcup\limits_{i=1}^{\nineq{D}}\calh_i^\complement,\ u_{k-1|t}\in\calu\\
    \forall k\in\Nint{t+1}{t+\Tplan}, &\  x_{k|t}\in\bigcup\limits_{i=1}^{\nineq{D}} \calk_i'(k,
    k+\Tsafe),\\
\end{array}\label{eq:prob_sketch_2}
\end{align}
\endgroup
with $\calk_i'(k, k+\Tsafe)\triangleq\calk(k, k+\Tsafe, \calh_i^\complement\ominus\Eoffnom)\ominus\Eoffnom$.
The non-convexity in \eqref{eq:prob_sketch_2} arises from the \emph{disjunctive
constraints}~\cite{disjunctive}.

For a sampling time $T_\text{sample}=20$ minutes, we solve \eqref{eq:prob_sketch} with a planning horizon of $2$ hours ($\Tplan=6$) and an abort-safety time horizon of $6$ hours ($\Tsafe=18$). 
We consider the nominal control set $\calu=\llbracket-1/3,1/3\rrbracket$ (in m/s), off-nominal input set $\Uoffnom=0.1\calu$, post-failure process noise in the ellipsoid $\Woffnom=(1/60 I_3, 0_{3\times 1})$ (in m/s), and post-failure measurement noise in the ellipsoid  $\Eoffnom=(\operatorname{diag}(1, 1, 1, 1/60, 1/60, 1/60), 0_{6\times 1})$ (in m and m/s). We define the origin-centered keep-out set $\cald={\llbracket-100,100\rrbracket}^3$ (in m) which contains the Lunar gateway~\cite{gateway}. We also define a proper cone characterized by four rays originating from the origin as the line-of-sight cone $\call=\{x\in\bbr^3 \mid [0, 1, -1;0, 1, 1;-1, 1,
0;1, 1, 0]x \leq 0_{4\times 1}\}$. 
We define an ellipsoidal target set
$\calt=(0.05 I_3, c_\text{target})$ with $c_\text{target}=[0.2416;
-0.4017; -0.1738]$ and initial state $x_0=[ 1.4498; 
   -2.4105;
   -1.0429;
    0.01;
    0.01;
    0.01]$ such that, from the origin, the target is $0.5$ km away and the initial state is $3$ km away with non-zero initial velocity.
    We rotate
    $\cald$ and $\call$ to have the $+\mathrm{y}$-face of $\cald$ and the axis of symmetry of $\call$ be aligned with the line segment joining $x_0$ and
    $c_\text{target}$ respectively. 
    The choice of parameters considers a rendezvous approach with the Sun behind the spacecraft, as the gateway flies near the apolune of the NRHO. 

    The exact computation of $\calk_i'(k,
    k+\Tsafe)$ based on polytopes~\cite{MPT3} is challenging,  due to the complexity of the calculations involved in the considered problem setting. Therefore, we use the proposed approach for inner-approximating $\Tsafe$-RC set using constrained zonotopes to enforce abort-safety constraint.
Specifically, we use Thm.~\ref{thm:pdiff_general} and Sec.~\ref{sub:set_recursion_no_inv_A} to compute constrained zonotopic, inner-approximations of
$\calk_i'(k,k+\Tsafe)$, and then use big-M formulations to cast the disjunctive constraints in \eqref{eq:prob_sketch_2} as
mixed-integer linear constraints. 

Consider $\nineq{D}$ constrained zonotopes $\{\calc_i\}_{i=1}^{\nineq{D}}$
    where $\calc_i\subseteq \calk_i'(k,k+\Tsafe)$ for each $i\in\Nint{1}{\nineq{D}}$ and some $k\in\Nint{t}{t+\Tplan}$.
Using $\nineq{D}$ auxiliary continuous variables
$\dvec_i\in\bbr^{\ndvec{C,i}}$ and binary variables
$\delta_i\in\{0,1\}$, and a sufficiently large $M>0$, the following set of $2(n+ \ndvec{C,i}) + \nconst{C,i} +1$ mixed-integer linear constraints is sufficient for $x_{k|t}\in\cup_{i=1}^{\nineq{D}}\calk_i'(k,k+\Tsafe)$ at any $k\in\Nint{t+1}{t+\Tplan}$, 
\begingroup
    \makeatletter\def\f@size{9}\check@mathfonts
\begin{align}
    \forall i\in\Nint{1}{\nineq{D}}, &\quad {\|G_{C,i}\dvec_i + c_{C,i} - x_{k|t}\|}_\infty\leq M(1 - \delta_i) \nonumber\\
    \forall i\in\Nint{1}{\nineq{D}}, &\quad A_{C,i}\dvec_i = b_{C,i},\  {\|\dvec_i\|}_\infty \leq 1,\nonumber \\
                               &\quad\sum\nolimits_{i=1}^{\nineq{D}} \delta_i \geq 1.\nonumber
\end{align}
\endgroup
Similar mixed-integer constraints based on big-M can be used to encode the disjunctive constraint $x_{k|t}\in\cup_{i=1}^{\nineq{D}}
\calh_i^\complement$~\cite{disjunctive}. Thus, \eqref{eq:prob_sketch_2} is a mixed-integer quadratic program, which can be solved via off-the-shelf solvers like \texttt{GUROBI}~\cite{GUROBI}.

\begin{figure}
        \centering
        \includegraphics[width=0.8\linewidth,trim={70 0 100 40},
    clip]{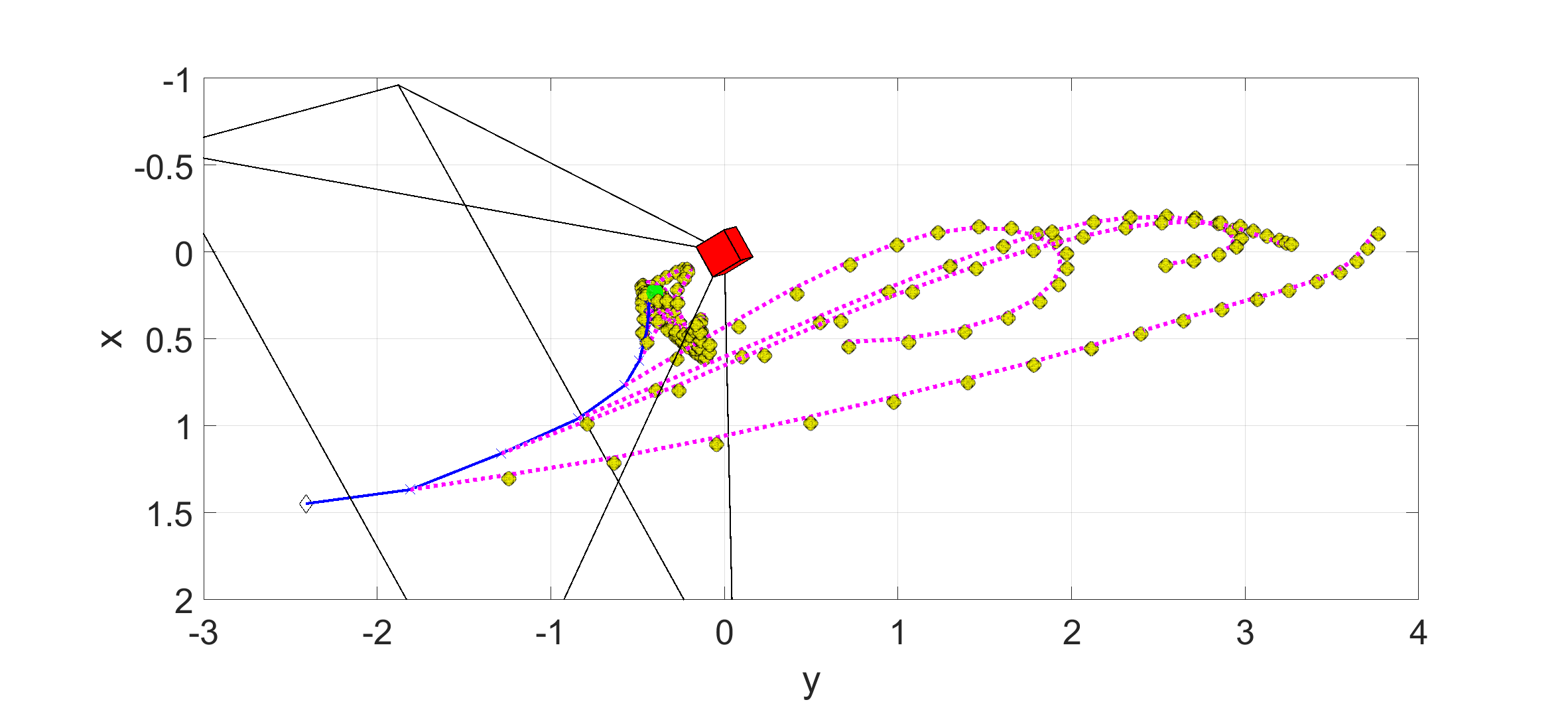}
    \caption{
    Designed nominal rendezvous trajectory along with the abort trajectories in case of failures at  $t\in\Nint{1}{\Tsim}$.
    The nominal trajectory starts at the initial state (diamond) and reaches the target set $\calt$ (green) in $\Tsim=10$ time steps, while staying within the line-of-sight cone $\call$ (black). The abort-safety requirement curves the nominal trajectory away from the keep-out set (red) at all times. The abort trajectories stay outside the
        keep-out set, despite the presence of disturbances which are adversarially chosen according to
        \eqref{eq:adv_dist}. 
        See \href{https://youtu.be/6BPmHgxD3OI}{https://youtu.be/6BPmHgxD3OI} for more details.
        }
    \label{fig:spacecraft_traj}
\end{figure}

\begin{figure}
    \centering
    \includegraphics[width=0.7\linewidth,trim={60 0 90 0},
    clip]{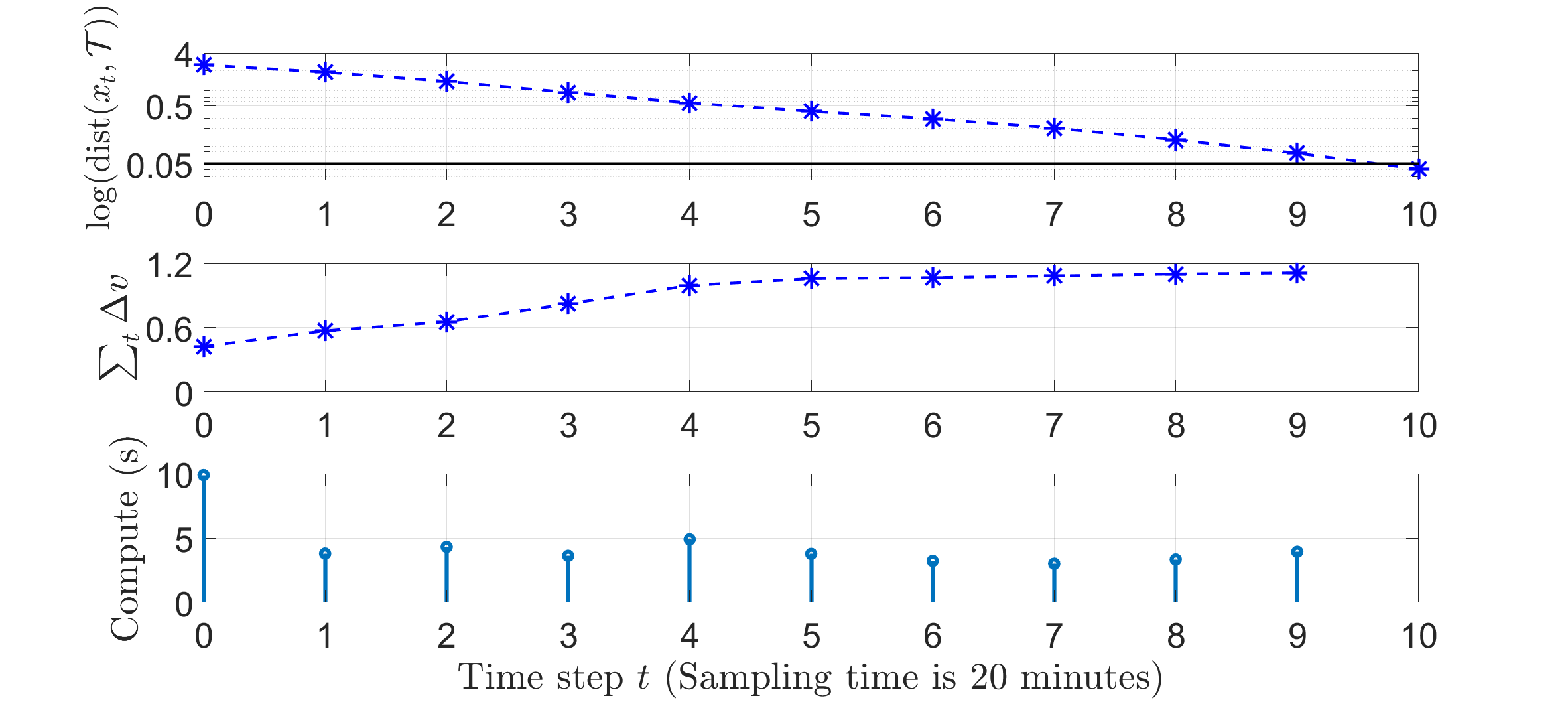}
    \caption{Evolution of the distance to the target (in log-scale) and cumulative $\Delta v$ over the course of
        rendezvous, and
        computation time for each solution of \eqref{eq:prob_sketch_2}.}
    \label{fig:spacecraft_metrics}
\end{figure}

We can also compute
an abort-safe control $u_{k|t}=\pi_k(\hat{z}_{k|t})$ 
for the current state estimate $\hat{z}_{k|t}$ by solving a convex problem,
\begingroup
    \makeatletter\def\f@size{9}\check@mathfonts
\begin{align}
    \begin{array}{rl}
    \underset{u_{k|t}\in\Uoffnom}{\text{min.}}&\ J(u_{k|t})\ \text{s.\ t.} \ A_k \hat{z}_{k|t} + B_k u_{k|t} \in \calk_{k+1|t}\ominus\Phi_{k},
    \end{array}\label{eq:pi_k_defn}%
\end{align}
\endgroup
where $J$ is a user-specified, convex cost function on the abort-safe control (we choose $J=\|\cdot\|_2$ to minimize post-failure fuel consumption), $\{\calk_{k|t}\}_{k=t}^{t+\Tsafe}$ is the sequence of sets obtained using
Sec.~\ref{sub:set_recursion_no_inv_A} with $\calk_{t+\Tsafe|t}=\calh_i^\complement\ominus\Eoffnom$
and $\calk(t, \Tsafe,
\calh_i^\complement\ominus\Eoffnom)=\calk_{t|t}$, as prescribed in Prop.~\ref{prop:suff_abort_safety}. We compute the corresponding adversarial disturbance $\phi_{k|t}$ by solving another convex
problem,
\begingroup
    \makeatletter\def\f@size{9}\check@mathfonts
\begin{align}
    \underset{\phi_{k|t}\in\Phi_k}{\text{min}} &\quad \operatorname{dist}(A_k \hat{z}_{k|t} + B_k \pi_k(\hat{z}_{k|t}) + \phi_{k|t},\cald).\label{eq:adv_dist}%
\end{align}
\endgroup
Here, \eqref{eq:adv_dist} computes $\phi_{k|t}$ that reduces the distance between the next state estimate and the keep-out set. 
Thus, we approximate the true worst-case disturbance sequence, whose exact computation would have required solving a two-player game, which is computationally difficult in six dimensions~\cite{bertsekas_1971}.

Fig.~\ref{fig:spacecraft_traj} demonstrates the designed nominal rendezvous trajectory, which takes $\Tsim = 10$ time steps ($200$
minutes) to reach $\calt$ with a total $\Delta v$ of $1.07$ m/s. The trajectory design required the computation of $385$
$18$-step RC sets, and the set computation took a total
of $7.53$ seconds. Fig.~\ref{fig:spacecraft_traj} also shows the designed abort-safe trajectories originating from each
time step of the nominal rendezvous trajectory using the post-failure control \eqref{eq:pi_k_defn} and adversarial disturbances \eqref{eq:adv_dist}, along with an
outer-approximation of the one-step forward
reach set $A_t \hat{z}_{k|t} + B_{k|t} \pi_k(\hat{z}_{k|t}) +
\Phi_k$. As expected, the abort-safety requirement \eqref{eq:abort_safety} is satisfied at all times
$t\in\Nint{1}{\Tsim}$. As the spacecraft approaches the target, we observe that the conservativeness of our enforcement of the abort-safety requirement \eqref{eq:abort_safety} via Prop.~\ref{prop:suff_abort_safety} causes the abort trajectories to cluster in front of the keep out set $\cald$. In actual rendezvous missions, the spacecraft would receive a go/no-go decision for docking as it nears the target.

Fig.~\ref{fig:spacecraft_metrics} shows the evolution of the distance to the target set (in log-scale) and the cumulative
$\Delta v$ expended over the course of the rendezvous, as well as the computation time spent solving
\eqref{eq:prob_sketch_2} at each time step $t\in\Nint{0}{\Tsim}$. The rendezvous trajectory initially uses
moderately high control inputs $\Delta v$ to steer the spacecraft towards the target while satisfying the
abort-safety requirements, and then utilizes the
momentum to reach the target set with minimal additional control inputs, as expected. The rendezvous trajectory maintains abort-safety using the constrained zonotope-based constraints computed with the method proposed in this paper.

\section{Conclusion}
\label{sec:conc}

We presented novel theory and algorithms to approximate the Pontryagin difference between a constrained zonotopic minuend and a symmetric, convex, and compact subtrahend. For broad classes of subtrahends, our approach admits closed-form expressions for the inner-approximation. We use these algorithms for a tractable and scalable computation of an inner-(and outer-)approximation of the robust controllable set for discrete-time linear systems with additive disturbance sets that are symmetric, convex, and compact and subject to linear state and input constraints. We showed by  numerical simulations that the proposed approach provides non-trivial inner-approximations of the RC sets with significantly shorter computation times than the previously published methods.

\appendix

\section{Proofs}
\label{app:proofs_aux}

\subsection{Proof of Prop.~\ref{prop:Indep}}
\label{app:proofs_aux_prop_indep}

    \emph{Proof of 1)} Assume for contradiction that, $G_C$ does not have full row rank for some  full-dimensional constrained zonotope $\calc=(G_C, c_C, A_C, b_C)$. 
    Then, there exists a vector $\alpha_{G_C}\in\bbr^n,\ \alpha_{G_C}\neq 0$ such that
    $\alpha_{G_C}^\top G_C=0$. 
    From \eqref{eq:b_infty}, for every $x\in\calc$, there exists $\dvec\in\calb_\infty(A_C,b_C)$ such that $G_C\dvec = x - c_C$. 
    Consequently, $\calc\subset \{x \mid \alpha_{G_C}^\top x = \alpha_{G_C}^\top c_C\}$ since $\alpha_{G_C}^\top(x - c_C)=\alpha_{G_C}^\top G_C \dvec=0$ for every $x\in\calc$.
    In other words, the affine dimension of $\calc$ is smaller than $n$.
    However, this is a contradiction since $\calc$ is full-dimensional. 
    Thus, $G_C$ has full row rank for every full-dimensional $\calc=(G_C, c_C, A_C, b_C)$. 

    \emph{Proof of 3)}
    We now show that Algo.~\ref{algo:min_row} converts any full-dimensional $\calc=(G_C, c_C, A_C, b_C)$ into a \textsc{MinRow} representation.
    From 1), $G_C'=G_C$ has full row rank. 
    Consider $A_C', b_C'$ in Step~\ref{step:ACprime} of Algo.~\ref{algo:min_row} with $[A_C',b_C']\in\bbr^{M_C'\times(N_C+1)}$.
    Then, $M_C'=\operatorname{rank}([A_C',b_C'])=\operatorname{rank}([A_C,b_C])\leq M_C$.
    It suffices to show that $\calc=(G_C,c_C,A_C',b_C')$ and $[G_C;A_C']$ has full row rank. 

    Without loss of generality, assume that $A_C',b_C'$ are the first $M_C'$ rows of $[A_C, b_C]$. 
    Since $\operatorname{rank}([A_C,b_C])=\operatorname{rank}([A_C',b_C'])$, every row of $[A_C,b_C]$ is a linear combination of the rows of $[A_C',b_C']$.
    In other words, there exists $E\in\bbr^{(\nconst{C}-\nconst{C}')\times\nconst{C}'}$ such that
    $[A_C,b_C]=[I_{\nconst{C}'};E][A_C',b_C']$. 
    The matrix $[I_{\nconst{C}'};E]$ has full column rank implying that $[I_{\nconst{C}'};E]y=0$ if and only if $y=0$~\cite[Ch. 11]{boyd_introduction_2018}. 
    Thus, 
    $\{\dvec \mid A_C\dvec=b_C\}=\{\dvec \mid A_C'\dvec=b_C'\}$, since for any $\dvec$ such that $A_C\dvec=b_C$,
    $[A_C,b_C][\dvec;-1]=0\Leftrightarrow[I_{\nconst{C}'};E][A_C',b_C'][\dvec;-1]=0\Leftrightarrow[A_C',b_C'][\dvec;-1]=0$. Therefore, $\calb_\infty(A_C,b_C)=\calb_\infty(A_C',b_C')$ and $\calc=(G_C,c_C,A_C,b_C)=(G_C,c_C,A_C',b_C')$.

    Since full-dimensional constrained zonotopes are non-empty, $\{\dvec \mid A_C'\dvec=b_C'\}$ is non-empty and $\operatorname{rank}(A_C')=\operatorname{rank}([A_C',b_C'])=M_C'$.
    Assume, for contradiction, that the matrix $[G_C;A_C']$ has linearly dependent rows. In other words, there
    exists $\beta_1\in\bbr^n,\ \beta_1\neq 0$ and $\beta_2\in\bbr^{\nconst{C}'},\ \beta_2\neq 0$ s.t.,%
    \begin{align}
        [\beta_1^\top,\beta_2^\top][G_C;A_C']=0 \equiv \beta_1^\top G_C + \beta_2^\top A_C'=0\label{eq:beta_1_beta_2}.
    \end{align}
    We know $\beta_1^\top G_C\neq0$ and $\beta_2^\top A_C'\neq 0$, since rows of $G_C$ and $A_C'$ are linearly
    independent among themselves. 
    From \eqref{eq:constrained_zonotope},
    for all $x\in\calc$, there
    exists $\dvec\in\bbr^{\ndvec{C}}$ s.t.,
    \begin{align}
        [G_C;A_C']\dvec = [x - c_C;b_C']\text{ and } \|\dvec\|_\infty \leq 1\label{eq:sop_form}.
    \end{align}
    By \eqref{eq:beta_1_beta_2} and \eqref{eq:sop_form}, $\calc\subset\{x \mid \beta_1^\top (x - c_C) + \beta_2^\top b_C'=0\}$, which is a contradiction 
    since $\calc$ is full-dimensional. Thus, all rows of $[G_C;A_C']$ are linearly independent, and $(G_C, c_C, A_C', b_C')$ is a \textsc{MinRow} representation of $\calc$.

    \emph{Proof of 2)}
    Since constrained zonotopes are a representation of polytopes~\cite[Thm. 1]{scott_constrained_2016}, the proof of 3) also shows that every full-dimensional polytope admits a \textsc{MinRow} representation.
    We now show that the reverse is also true, i.e., any non-empty polytope $\calc$ in \textsc{MinRow} representation $(G_C, c_C, A_C, b_C)$ is full-dimensional.
    Assume, for contradiction, that $\calc$ is not full-dimensional.
    Then, there exists an affine set characterized by $\alpha\in\bbr^n,\alpha\neq 0$ and $\beta\in\bbr$ such that $\calc\subset\{x \mid \alpha^\top x = \beta\}$.
    Since $\calc$ is non-empty,
    $\calc=\calc\cap\{x \mid \alpha^\top x = \beta\}$.
    From~\cite{raghuraman2022set}, $\calc=(G_C,c_C, [A_C;\alpha^\top G_C], [b_C;\beta-\alpha^\top c_C])=(G_C,c_C, A_C, b_C)$.
    Consequently, 
    there exists some $\gamma\in\bbr^{\nconst{C}},\gamma\neq 0$ such that $\alpha^\top G_C = \gamma^\top A_C$ and $\beta - \alpha^\top c_C = \gamma^\top b_C$, a contradiction since $[G_C;A_C]$ has full row rank. 
    Thus, $\calc$ is full-dimensional, which completes the proof.

    \emph{Proof of 4)}
    Algo.~\ref{algo:invertible} follows the steps of~\cite[Thm. 1]{scott_constrained_2016} with the zonotope $\calz$ in Step 1 defined as a rectangular outer-approximation with $\calc\subseteq \{x\in\bbr^n \mid l\leq
    x\leq u\}$, and $\sigma\in\bbr^{\nineq{C}}$ defined in Step 2 using \eqref{eq:support} such that $\calp=\{x\in\calz \mid \sigma \leq Hx \leq k\}\subset\bbr^n$.
    Thus, $(G_C, c_C, A_C, b_C)$ defined in Step 3 is a constrained zonotope representation of the given H-Rep polytope $\calc$. 
Since $\calc$ is full-dimensional, $l<u$ and $\sigma < k$. Consequently,  
\begingroup
    \makeatletter\def\f@size{9.5}\check@mathfonts
\begin{align}
    \left[\begin{array}{c}
    G_C\\
    A_C
    \end{array}\right]=\left[\begin{array}{cc}
    \operatorname{diag}\left({\frac{u - l}{2}}\right) & 0_{n\times M}\\
    H\operatorname{diag}\left({\frac{u - l}{2}}\right) & \operatorname{diag}\left({\frac{\sigma - k}{2}}\right)
    \end{array}\right],
\end{align}
\endgroup
is a lower triangular matrix with non-zero diagonal entries (and thereby, invertible). This completes the proof.
    
    \emph{Proof of 5)} Any \textsc{Invertible} representation satisfies
    \begin{align}
        \calc&=\left\{x\middle\vert\exists \dvec\in\bbr^{\ndvec{C}},\  
                        {\|\dvec\|}_\infty \leq 1,\ 
                        \left[G_C;A_C\right]\dvec=\left[x-c_C;b_C\right]
        \right\}\nonumber\\
        &=\left\{x\middle\vert 
                        {\left\| 
                        {\left[G_C;A_C\right]}^{-1}\left[x-c_C;b_C\right]\right\|}_\infty \leq 1
        \right\}.\label{eq:inf_norm_from_cz}
    \end{align}
    We obtain \eqref{eq:hrep_from_cz} by rearranging terms in \eqref{eq:inf_norm_from_cz}.
    \hfill\qed

\subsection{Proof of Prop.~\ref{prop:reverse}}
\label{app:proofs_aux_prop_reverse}

    Since $\calc$ is full-dimensional, it is non-empty.
    From \eqref{eq:constrained_zonotope} and strong duality (via refined Slater's condition~\cite{boyd2004convex}), %
    {\begingroup
    \makeatletter\def\f@size{9.5}\check@mathfonts
    \begin{align}
        \calc
        &=\left\{x \middle\vert \sup_{\substack{\nu\in\bbr^{n+\nconst{C}}\\
        {\|[G_C;A_C]^\top\nu\|}_1\leq 1}} \nu^\top[x - c_C;b_C] \leq 1\right\}.\label{eq:cz_duality}
    \end{align}
    \endgroup}%
    For
    every $i\in\Nint{1}{\ndvec{C}}$, define $\nu_i^\ast$ as the minimum-norm solution to $[G_C;A_C]^\top\nu =
    e_i$. 
    By Prop.~\ref{prop:Indep}.2, 
    $\nu_i^\ast$ is available in closed-form, where 
    {\begingroup
    \makeatletter\def\f@size{9}\check@mathfonts
    \begin{align}
        \nu_i^\ast &=
        {\left({[G_C;A_C][G_C;A_C]^\top}\right)}^{-1}[G_C;A_C]e_i ={\left({[G_C;A_C]}^\dagger\right)}^\top
                          e_i, \nonumber
    \end{align}
    \endgroup}%
    for every $i\in\Nint{1}{\ndvec{C}}$.
    Step~\ref{step:nu_i_defn} of Algo.~\ref{algo:outer_approx_hrep} (see \eqref{eq:nu_i_defn}) rescales $\nu_i^\ast$ to ensure
    that $[G_C;A_C]^\top\nu_i^\ast$ has a unit
    $\ell_1$-norm to obtain $v_i^\top$. Thus, $\pm v_i^\top$ are feasible
    for \eqref{eq:cz_duality} for
    every $i\in\Nint{1}{\ndvec{C}}$, and using \eqref{eq:cz_duality}, 
    \begin{align}
        \calc&\subseteq\calp\triangleq\left\{x\ \middle\vert\ {{\left\|V^\top [x - c_C;b_C]\right\|}_\infty \leq
        1}\right\},\label{eq:cz_outer_polytope}
    \end{align}
    with $V=[\nu^\ast_1;\ldots;\nu^\ast_{\ndvec{C}}]$.
    Steps~\ref{step:H_defn},~\ref{step:k_defn} of Algo.~\ref{algo:outer_approx_hrep} define $\calp=\{x \mid Hx \leq k\}$ with  
    $H\in\bbr^{2\ndvec{C}\times n}$ and $k\in\bbr^{2\ndvec{C}}$.\hfill\qed

\subsection{Proof of Prop.~\ref{prop:exact_reverse}}
\label{app:proofs_aux_prop_exact_reverse}

From \eqref{eq:inf_norm_from_cz}, an \textsc{Invertible} representation $\calc$ may also be expressed as the following H-Rep polytope,
\begin{align}
    \calc &=\left\{x\middle\vert \begin{array}{c}
              \forall i\in\Nint{1}{\ndvec{C}},\ \forall \delta\in\{-1,1\},\\ \delta e_i^\top {[G_C;A_C]}^{-1}[x-c_C;b_C]
      \leq 1\end{array}\right\}\label{eq:C_exact_i}.
\end{align}
Here, $v_i$ defined in Step~\ref{step:nu_i_defn} of
Algo.~\ref{algo:outer_approx_hrep} also simplifies to $e_i^\top {[G_C;A_C]}^{-1}$ when $[G_C;A_C]$ is
invertible~\cite[Sec. 11.5]{boyd_introduction_2018}.
Thus, Algo.~\ref{algo:outer_approx_hrep} returns a H-Rep polytope $\calp$ that coincides with \eqref{eq:C_exact_i} (thereby, equal to $\calc$).  The proof is completed with the observation that \eqref{eq:C_exact_i} is identical to \eqref{eq:inf_norm_from_cz}.\hfill\qed

\subsection{Proof of Prop.~\ref{prop:exact_general}}
\label{app:proofs_aux_prop_exact_general}

\emph{Exactness of $\calm^+$)} 
As seen from Prop.~\ref{prop:exact_reverse}, Algo.~\ref{algo:outer_approx_hrep} computes an exact H-Rep polytope, given an \textsc{Invertible}
representation $\calc$.
Since the rest of steps of
Algo.~\ref{algo:outer_approx} are exact, $\calm^+=\calc\ominus\cals$.

\emph{Exactness of $\calm^-$)}
    For an \textsc{Invertible} representation $\calc$, $\Gamma$ prescribed by Prop.~\ref{prop:pdiff_general_l2} is
    \begin{align}
        \Gamma=[G_C;A_C]^{-1}[I_n;0_{\nconst{C}\times n}]\label{eq:Gamma_under_Assum.2},
    \end{align}
    by~\cite[Sec. 11.5]{boyd_introduction_2018}. 
    We show that $\calm=\calc\ominus\cals$ is an affine transformation of $\calc_L$ defined in
    the proof of Thm.~\ref{thm:pdiff_general} (see \eqref{eq:m_minus_defn_general}). 

    {\begingroup
    \makeatletter\def\f@size{8.5}\check@mathfonts
        \begin{align}
          &\calm=\left\{x\middle\vert \begin{array}{c}
              \ \forall s\in\cals_0,\ \forall i\in\Nint{1}{\ndvec{C}},\ \forall \delta\in\{-1,1\},\\ \delta e_i^\top
              {[G_C;A_C]}^{-1}[x + s + c_S -c_C;b_C]
      \leq 1\end{array}\right\}\label{eq:exact_interim_2}
      \\
              &=\left\{x\ \middle\vert\ 
        \begin{array}{c}
            \forall s\in\cals_0,\ \forall i\in\Nint{1}{\ndvec{C}},\ \forall \delta \in\{-1,1\},\\
            \delta e_i^\top \left({\Gamma x+\Gamma s + {[G_C;A_C]}^{-1}[- c_{M^-};b_C]}\right)\leq 1\\
        \end{array}\right\}\label{eq:exact_interim_3}\\
              &=\left\{x\ \middle\vert\ 
        \begin{array}{c}
            \forall i\in\Nint{1}{\ndvec{C}},\ \forall \delta \in\{-1,1\},\\
            \delta e_i^\top \left({\Gamma x + {[G_C;A_C]}^{-1}[- c_{M^-};b_C]}\right)\leq 1
            - \rho_{\cals_0}(\Gamma^\top e_i)\\
        \end{array}\right\}\label{eq:exact_interim_4}\\
              &=\left\{x\ \middle\vert\ 
        \begin{array}{c}
\dvec \triangleq \Gamma x + {[G_C;A_C]}^{-1}[- c_{M^-};b_C],\\
            \forall i\in\Nint{1}{\ndvec{C}},\ \forall \delta \in\{-1,1\},\ \delta e_i^\top \dvec \leq D_{ii}\\
        \end{array}\right\}\label{eq:exact_interim_5}\\
              &=\left\{x\ \middle\vert\ 
        \begin{array}{c}
            G_C\dvec + c_{M^-} = x,\ A_C\dvec = b_C,\\
            \forall i\in\Nint{1}{\ndvec{C}},\ \forall \delta \in\{-1,1\},\ \delta e_i^\top \dvec \leq D_{ii}\\
        \end{array}\right\}\label{eq:exact_interim_6}\\
              &=\left\{G_C\dvec + c_{M^-}\ \middle\vert\ 
            A_C\dvec = b_C,\ \forall i\in\Nint{1}{\ndvec{C}},\ |e_i^\top \dvec| \leq D_{ii}
        \right\} \nonumber \\
              &=G_C\calc_L + c_{M^-}=\calm^-\label{eq:exact_interim_8}.
    \end{align}%
    \endgroup}%
    Here, 
    \eqref{eq:exact_interim_2} follows from \eqref{eq:pdiff} and \eqref{eq:C_exact_i} in
    Sec.~\ref{app:proofs_aux_prop_exact_reverse},
    \eqref{eq:exact_interim_3} follows from \eqref{eq:Gamma_under_Assum.2}, \eqref{eq:exact_interim_4}
    follows from encoding the condition for all $s\in\cals_0$ using the support function
    of $\cals_0$, \eqref{eq:exact_interim_5} follows from the
    definition of $D_{ii}$ in \eqref{eq:diag_matrix_general}, \eqref{eq:exact_interim_6} follows from
    expressing $\dvec$ as a solution of linear equations $[G_C;A_C]\dvec=[x - c_{M^-};b_C]$ with invertible
    $[G_C;A_C]$, and \eqref{eq:exact_interim_8} follows from
    the definition of $\calc_L$ in \eqref{eq:general_interim_8} (see the proof of Thm.~\ref{thm:pdiff_general}). Consequently, for an \textsc{Invertible} representation $\calc$, Prop.~\ref{prop:pdiff_general_l2} provides an exact
    characterization of $\calc\ominus\cals$.\hfill\qed

\subsection{Proof of Prop.~\ref{prop:suff_abort_safety}}
\label{app:proofs_aux_prop_suff_abort_safety}

    From~\cite[Thm. 2.1(iii)]{kolmanovsky1998theory}, $(\calu\ominus\calv)\oplus\calv\subseteq \calu$ for any sets $\calu,\calv$.
    Let a failure occur at some time $s\in\bbn$. Using
\eqref{eq:z_t_hat_z_t} and~\cite[Thm. 2.1(iii)]{kolmanovsky1998theory},  for some
    $i\in\Nint{1}{\nineq{D}}$, $\hat{z}_{s|s}=x_s\in\calk(s, \Tsafe,
    \calh_i^\complement\ominus\Eoffnom)\ominus\Eoffnom$, which implies that $\hat{z}_{s|s}\in\calk(s, \Tsafe,
    \calh_i^\complement\ominus\Eoffnom)$.
    From Defn.~\ref{defn:RC_set}, for every time step $k\in\Nint{s}{s+\Tsafe}$, there exists 
    $\pi_k\in\Pi$ such
    that $u_{k|s}=\pi_k(\hat{z}_{k|s})\in\Uoffnom$ steers the state estimate according to the dynamics
    \eqref{eq:offnom_ltv_est} to satisfy $\hat{z}_{k|s}\in
    \calh_i^\complement \ominus\Eoffnom$, despite the disturbance $\phi_{k|s}\in\Phi_k$. 
    By \eqref{eq:z_t_hat_z_t} and~\cite[Thm. 2.1(iii)]{kolmanovsky1998theory},  $\hat{z}_{k|s}\in \calh_i^\complement
    \ominus\Eoffnom$ implies that $z_{k|s}\in \calh_i^\complement$ (and thereby, $z_{k|s}\not\in\cald$) for all
    $k\in\Nint{s}{s+\Tsafe}$, which completes the proof.\hfill\qed

\bibliographystyle{ieeetr}        %
\bibliography{refs}
\end{document}